\documentclass[11pt,letterpaper]{amsart}
\usepackage[margin=1in]{geometry}
\usepackage{hyperref}
\usepackage{cancel}
\hypersetup{
    colorlinks=true,       
    linkcolor=red,          
    citecolor=blue,        
    filecolor=magenta,      
    urlcolor=blue           
}
\usepackage{amsmath, amssymb, amscd, amsfonts, amsthm, mathrsfs, graphicx}
\usepackage[arrow,matrix,curve,cmtip,ps]{xy}
\usepackage{enumitem}
\usepackage{pinlabel}
\usepackage{tikz-cd}
\usepackage{color}
\usepackage{xcolor}
\usepackage{subcaption}
\usepackage{enumitem}

\usepackage[normalem]{ulem}

\DeclareGraphicsExtensions{.png,.pdf,.jpg,.jpeg}

\usepackage{centernot}

\usepackage{pinlabel}	

\makeatletter
\newtheoremstyle{theorem-giventitle}
        {}{}              
        {\itshape}                      
        {}                              
        {\bfseries}                     
        {.}                             
        {\thm@headsep}                             
        {\thmnote{\bfseries#3}}
\newtheoremstyle{theorem-givenlabel}
        {}{}              
        {\itshape}                      
        {}                              
        {\bfseries}                     
        {.}                             
        {\thm@headsep}                             
        {\thmname{#1}~\thmnumber{#3}\setcurrentlabel{#3}}
\newtheoremstyle{definition-giventitle}
        {}{}              
        {}                      
        {}                              
        {\bfseries}                     
        {.}                             
        {\thm@headsep}                             
        {\thmnote{\bfseries#3}}
\def\setcurrentlabel#1{\gdef\@currentlabel{#1}}
\makeatother

\newtheorem{theorem}{Theorem}[section]

\newtheorem{proposition}[theorem]{Proposition}
\newtheorem{corollary}[theorem]{Corollary}
\newtheorem{lemma}[theorem]{Lemma}

\newtheorem*{theorem*}{Theorem}
\newtheorem*{proposition*}{Proposition}
\newtheorem*{lemma*}{Lemma}
\newtheorem*{corollary*}{Corollary}
\newtheorem*{rep@theorem}{\rep@title}
\newtheorem*{claim*}{Claim}
\newcommand{\newreptheorem}[2]{
\newenvironment{rep#1}[1]{
\def\rep@title{#2 \ref{##1}}
\begin{rep@theorem}}
{\end{rep@theorem}}}
\makeatother
\theoremstyle{definition}
\newtheorem{definition}[theorem]{Definition}
\newtheorem{remark}[theorem]{Remark}
\newtheorem{example}[theorem]{Example}

\newtheorem*{question*}{Question}

\newreptheorem{theorem}{Theorem}
\newreptheorem{lemma}{Lemma}
\newreptheorem{proposition}{Proposition}
\newreptheorem{corollary}{Corollary}

\newcommand{\bdry}{\partial}

\DeclareMathOperator{\PSL}{PSL}

\newcommand{\Z}{\mathbb{Z}}

\newcommand{\CC}{\mathbb{C}}

\newcommand{\K}{\mathcal{K}}

\newcommand{\mL}{\mathcal{L}}
\newcommand{\U}{\mathcal{U}}
\newcommand{\im}{\operatorname{Im}}
\newcommand{\lk}{\operatorname{lk}}

\newcommand{\Id}{\operatorname{Id}}

\newcommand{\G}{\mathcal{G}}

\DeclareMathOperator{\HE}{HE}
\DeclareMathOperator{\SO}{SO}
\newcommand{\Homeo}{\text{Homeo}}
\newcommand{\HMod}{\text{HMod}}
\newcommand{\Aut}{\text{Aut}}
\newcommand{\Out}{\text{Out}}
\newcommand{\Inn}{\text{Inn}}
\newcommand{\Mod}{\text{Mod}}

\newcommand{\Diff}{\text{Diff}}
\DeclareMathOperator{\Isom}{Isom}
\newcommand{\Hy}{\mathbb{H}}
\newcommand{\Sa}{\mathbb{S}}

\newcommand{\overl}{\overline}

\newcommand{\R}{\mathbb{R}}

\renewcommand{\S}{\mathscr{S}}

\newcommand{\into}{\hookrightarrow}

\newcommand{\Sum}{\displaystyle \sum}

\newcommand{\pref}[1]{(\ref{#1})}

\usepackage{letltxmacro}
\LetLtxMacro\Oldfootnote\footnote

\begin{document}

\title{Isotopy and equivalence of knots in $3$-manifolds}

\author[Aceto]{Paolo Aceto}
\address{Laboratorie Paul Painlevé - Université de Lille, France}
\email{paoloaceto@gmail.com}
\urladdr{https://sites.google.com/view/paoloaceto/home}

\author[Bregman]{Corey Bregman}
\address{Department of Mathematics, Tufts University, MA, USA}
\email{corey.bregman@tufts.edu}
\urladdr{https://sites.google.com/view/cbregman}

\author[Davis]{Christopher W.\ Davis}
\address{Department of Mathematics, University of Wisconsin--Eau Claire, WI, USA}
\email{daviscw@uwec.edu}
\urladdr{http://people.uwec.edu/daviscw}

\author[Park]{\\JungHwan Park}
\address{Department of Mathematical Sciences, KAIST, South Korea}
\email{jungpark0817@gmail.com}
\urladdr{https://mathsci.kaist.ac.kr/~jungpark0817}

\author[Ray]{Arunima Ray}
\address{School of Mathematics and Statistics, University of Melbourne, VIC, Australia}
\email{aru.ray@unimelb.edu.au}
\urladdr{https://aru-ray.github.io/}

\date{\today}

\def\subjclassname{\textup{2020} Mathematics Subject Classification}
\expandafter\let\csname subjclassname@1991\endcsname=\subjclassname

\begin{abstract}
Two knots $K$ and $J$ in $S^3$ are isotopic if and only if they are related by an orientation-preserving diffeomorphism of $S^3$. This claim follows from the fact that any orientation-preserving self diffeomorphism of $S^3$ is isotopic to the identity.   We show that this same idea applies to any prime oriented closed 3-manifold.  More precisely, 
we show that a prime closed oriented 3-manifold contains a pair of
equivalent but non-isotopic knots if and only if the (orientation-preserving) mapping class group is non-trivial. 
When $M$ is additionally irreducible we show that an orientation-preserving diffeomorphism of $M$ is isotopic to the identity if and only if it preserves all homotopy classes of knots.  For knots in $S^1\times S^2$ (the only reducible prime oriented 3-manifold) we exhibit infinitely many knots whose isotopy classes are not preserved by the Gluck twist.
\end{abstract}

\maketitle

\section{Introduction}

Let $M$ be a smooth, closed, oriented  $3$-manifold. A knot $K$ in $M$ is a smooth embedding of the (oriented) circle in $M$. There are three natural equivalence relations on knots in $M$. We say that the knots $K,J\colon S^1\into M$ are \emph{equivalent} if there is an orientation-preserving diffeomorphism $f\colon M\to M$ such that $f\circ K=J$. On the other hand, we say that $K$ and $J$ are \emph{isotopic} if there exists a smooth map
\[
F \colon S^1 \times [0,1] \to M; \qquad (x,t) \mapsto F_t(x),
\]
such that each $F_t$, for $t\in [0,1]$, is a smooth embedding, with $F_0 = K$ and $F_1 = J$. Finally we say that $K$ and $J$ are \emph{ambient isotopic} if there exists a smooth map
\[
G \colon M \times [0,1] \to M; \qquad (x,t) \mapsto G_t(x),
\]
such that each $G_t$, for $t\in [0,1]$, is a diffeomorphism, with $G_0=\Id_M$ and $G_1 \circ K = J$. 
By the isotopy extension theorem, two knots in any fixed $M$ are isotopic if and only if they are ambient isotopic. It is clear from the definitions that ambient isotopy implies equivalence of knots. It is then interesting to ask to what extent equivalence implies isotopy for knots.

Let $\Mod^+(M)$ denote the \emph{mapping class group} of $M$, that is, the set of orientation-preserving diffeomorphisms of the smooth, closed, oriented $3$-manifold $M$, modulo isotopy.\footnote{While it is common to define the mapping class group using homeomorphisms modulo isotopy, in this article we work in the smooth category for simplicity. This is justified by the fact that the topological and smooth mapping class groups are naturally isomorphic. See Section~\ref{sec:proof-irred} for further details.} 
Every orientation-preserving diffeomorphism of $S^3$ is isotopic to the identity~\cite{Cerf:1968}; in other words, the mapping class group of $S^3$ is trivial, and thus, the notions of equivalence and isotopy of knots coincide. Plainly, equivalent knots are isotopic in any $3$-manifold with trivial mapping class group. The converse is the main result of this paper. 

\begin{theorem}\label{thm:main}
Let $M$ be a prime, smooth, closed, oriented $3$-manifold. If $f\colon M\to M$ is an orientation-preserving diffeomorphism which is not isotopic to the identity, then there exists a knot $K\subseteq M$ such that $K$ and $f(K)$ are not isotopic.

In particular, $M$ contains a pair of equivalent but not isotopic knots if and only if $\Mod^+(M)$ is non-trivial. 
\end{theorem}

In the case of irreducible $3$-manifolds, Theorem~\ref{thm:main} is a direct consequence of the following more general theorem. Throughout the article, a loop $\gamma$ in a $3$-manifold $M$ refers to a continuous map $\gamma\colon S^{1}\to M$. The homotopy class of $\gamma$ canonically determines a conjugacy class in $\pi_{1}(M)$.

\begin{theorem}\label{thm:irred-rigid}\label{thm:freehomotopy}
Let $M$ be an irreducible, smooth, closed, oriented $3$-manifold.  If $f\colon M\to M$ is an orientation-preserving diffeomorphism which is not isotopic to the identity, then there exists a loop $\gamma\subseteq M$ such that $\gamma$ and $f(\gamma)$ are not homotopic.
\end{theorem}

Equivalently, Theorem~\ref{thm:irred-rigid} says that if $f$ acts by the identity on the set of conjugacy classes of $\pi_1(M)$ then $f$ is isotopic to the identity. Theorem~\ref{thm:irred-rigid} implies Theorem~\ref{thm:main} for irreducible, closed, oriented $3$-manifolds, since isotopic knots are in particular homotopic. 

We now describe the proof of Theorem~\ref{thm:irred-rigid}. As usual, given a group $G$, the outer automorphism group $\Out(G)$ is the quotient $\Aut(G)/\Inn(G)$. Each self-diffeomorphism of a $3$-manifold $M$ induces an element of $\Out(\pi_1(M))$. Here, we must quotient out by the inner automorphism group since the diffeomorphism may not preserve the basepoint of the fundamental group. In particular, since an isotopy may not preserve the basepoint, a self-diffeomorphism of $M$ that is isotopic to the identity can only be said to induce the trivial element of $\Out(\pi_1(M))$, rather than the identity map on $\pi_1(M)$ itself. 

For irreducible, oriented $3$-manifolds, it is a consequence of the work of many authors, as explained in the proof of Theorem~\ref{thm:mappingclass},
that the above map is an injection, that is, $\Mod^+(M)\hookrightarrow \Out(\pi_1(M))$. Then the proof of Theorem~\ref{thm:irred-rigid} would be completed by showing that if an orientation-preserving diffeomorphism $f\colon M\to M$ on an irreducible, closed, oriented $3$-manifold $M$ preserves homotopy classes of loops, or equivalently, fixes every conjugacy class of $\pi_1(M)$, then $f_*$ is an inner automorphism. 

\begin{definition}[Property~A of groups]\label{def:propA}
An automorphism of a group $G$ is called \textit{class preserving} if it fixes every conjugacy class of $G$. A group is said to have \textit{Grossman's Property A} if every class preserving automorphism is inner.\footnote{This notion should not be confused with the weakening of amenability, also called Property A, introduced by Yu in~\cite{otherproperty}.}
\end{definition}

Clearly, inner automorphisms are class preserving.  
Grossman introduced Property A in~\cite{Gr75} and showed that free groups and surface groups have Property A.  She then used this to prove that the outer automorphism groups of free groups and surface groups are residually finite. 

Although infinite groups without Property A have been constructed (see e.g.~\cite{Segal, KimZhou}), Property A has been shown to hold when a group exhibits some degree of hyperbolicity.  First, Neshchadim proved that non-trivial free products have Property A~\cite{Neshchadim:1996}. This was greatly generalized by Minasyan--Osin to the setting of relatively hyperbolic groups, who in fact showed that for such groups any automorphism that preserves normal subgroups is inner~\cite{MO10}. Most recently, Antol\'in--Minasyan--Sisto~\cite{AMS2016} further extended the results of~\cite{MO10} to commensurating endomorphisms of acylindrically hyperbolic groups.  

We prove the following result, which completes the proof of Theorem~\ref{thm:irred-rigid}.

\begin{theorem}\label{thm:PropertyA}Let $M$ be a closed, orientable $3$-manifold. Then $\pi_1(M)$ has Property~A.  
\end{theorem}

This theorem was already known in certain cases. Compact 3-manifold groups are either acylindrically hyperbolic, Seifert-fibered, or virtually solvable \cite[Corollary~2.9]{MO2015}, and they are known to be conjugacy separable~\cite{HWZ13}. Here we say that a group is \emph{Seifert-fibered} if it is the fundamental group of a Seifert-fibered $3$-manifold. Antol\'in--Minasyan--Sisto's proof~\cite{AMS2016} that $\Out(\pi_1(M))$ is residually finite when $M$ is a compact 3-manifold in the acylindrically hyperbolic case uses Grossman's criterion and therefore establishes that they have Property A. They also prove that $\Out(\pi_1(M))$ is residually finite in the Seifert-fibered or virtually solvable cases, but the method is more direct and does not address Property A.

On the other hand, Allenby--Kim--Tang have shown that if $M$ is a Seifert-fibered $3$-manifold whose base orbifold is not a sphere with precisely $3$ cone points nor a torus with a single cone point, then $\pi_1(M)$ has Property A~\cite{AKT2003, ATK:2009}. In Theorem~\ref{thm:SeifertA} we establish Property A for fundamental groups of all Seifert-fibered $3$-manifolds without any restriction on the base of the fibering. Finally, the case of virtually solvable but not Seifert-fibered corresponds to those closed manifolds supporting Sol-geometry.  We treat Property A for these manifolds in Proposition~\ref{prop:AnosovA}, completing the proof of Theorem~\ref{thm:PropertyA}.

In light of Theorem~\ref{thm:freehomotopy}, the only remaining case needed to verify Theorem~\ref{thm:main} is that of $S^1\times S^2$.
The \emph{geometric winding number} of a knot $K\subseteq S^1\times S^2$ is the minimal number of times that $K$ intersects a non-separating sphere in $S^1\times S^2$ and its \emph{$($algebraic$)$ winding number} is the algebraic intersection number of $K$ and $\{pt\}\times S^2$. The latter quantity is a well-defined integer since 
both $K$ and
$\{pt\}\times S^2$ are oriented. The manifold $S^1\times S^2$ is the result of $0$-framed Dehn surgery on $S^3$ along the unknot. Thus, every knot $K\subseteq S^1\times S^2$ may be presented as a knot in such a surgery diagram as in Figure~\ref{fig:Generic-Twist}. The orientation of $\{pt\}\times S^2$ is indicated by the arrow on the $0$-framed unknot.

Define the \emph{Gluck twist} 
\[
\G\colon S^1\times S^2\to S^1\times S^2; \qquad (\theta, s) \mapsto (\theta, \rho_\theta(s)),
\]
where $\rho_\theta$ is given by rotating $S^2$ about the vertical axis by angle $\theta$. For a given diagram of a knot $K\subseteq S^1\times S^2$, the effect of $\G$ is to insert a full right-handed twist in all strands intersecting $\{pt\}\times S^2$, as shown in Figure~\ref{fig:Generic-Twist}. Note that for any knot $K\subseteq S^1\times S^2$ the resulting knot $\G(K)$ is equivalent to $K$, by definition. For future reference, we observe that the Gluck twist can be isotoped to be supported on a small bicollar neighbourhood of $\{*\} \times S^2\subseteq S^1\times S^2$. Via this observation it becomes possible to define Gluck twist diffeomorphisms on arbitrary 3-manifolds, by twisting in a neighbourhood of an embedded 2-sphere.

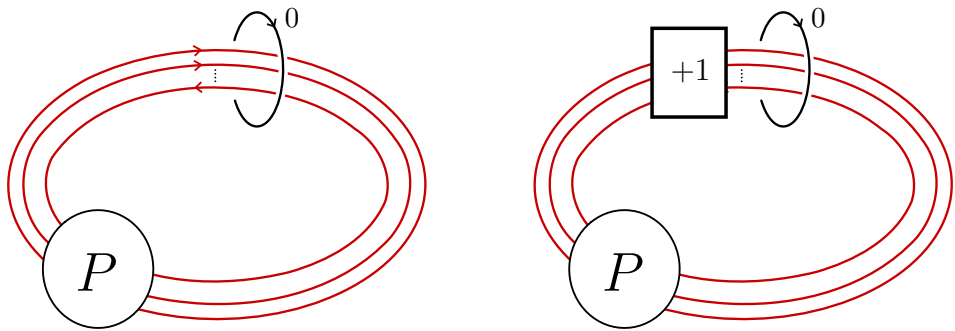
\begin{figure}[htb]
\begin{tikzpicture}
\node at (0,0) {\includegraphics[width=.9\textwidth]{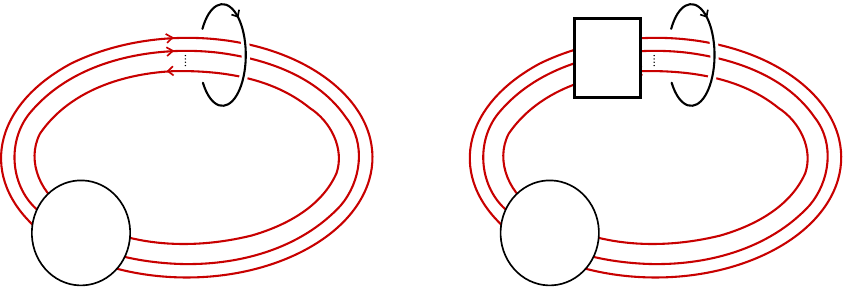}};
\node at (-6,-1.6) {\huge $P$};
\node at (2.28,-1.6) {\huge $P$};
\node at (-2.7,2.2) {$0$};
\node at (5.5,2.2) {$0$};
\node at (3.3,1.5) {\large $+1$};
\end{tikzpicture}
\caption{Left: A knot $K$ in $S^1\times S^2$.  Right: The image $\G(K)$ of $K$ under the Gluck twist $\G$.  The ``$+1$'' indicates one full right-handed twist.}\label{fig:Generic-Twist}\end{figure}

The mapping class group of $S^1\times S^2$ was shown to be $\Z/2\oplus \Z/2$\ by Gluck~\cite[Theorem 5.1]{Gluck-1962}, where the first $\Z/2$-factor is generated by the diffeomorphism given by reflection on both the $S^1$- and $S^2$-factors.
and the second is generated by the Gluck twist. The former reverses the orientation of the knot $S^1\times \{pt\}$. In contrast, the Gluck twist acts by the identity on homotopy classes of loops. Indeed, for knots in $S^1\times S^2$ with geometric winding number $0, 1$, or $2$, the Gluck twist preserves isotopy classes, since the effect of a Gluck twist is either trivial or can be undone by a slide (in particular, an isotopy) over a non-separating $2$-sphere (see Figure~\ref{fig:Wind-Num-2}). 
The same argument shows that for any even integer $w=2k$, the $(k,1)$-cable $J$ of the knot of Figure~\ref{fig:Wind-Num-2} has  winding number $w$ and is isotopic to $\G(J)$. Nevertheless, and in contrast to Theorem~\ref{thm:irred-rigid}, we show that there exist knots in $S^1\times S^2$ whose isotopy classes are changed under the Gluck twist.  We show a similar result for $\#^n (S^1\times S^2)$ in Corollary~\ref{cor:nS1xS2}.

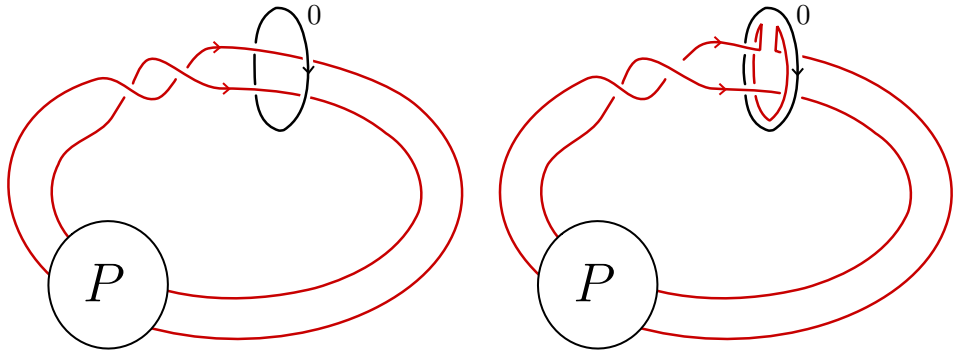
\begin{figure}[htb]
\begin{tikzpicture}
\node at (0,0) {\includegraphics[width=.9\textwidth]{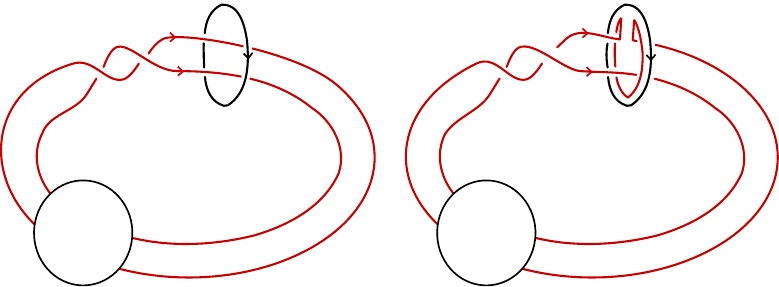}};
\node at (-5.8,-1.7) {\huge $P$};
\node at (1.9,-1.7) {\huge $P$};
\node at (-2.3,2.25) {$0$};
\node at (5.4,2.25) {$0$};
\end{tikzpicture}
\caption{Left: $\G(K)$ for some geometric winding number two knot $K$.  Right: A slide over the $0$-framed curve reduces $\G(K)$ back to $K$. }\label{fig:Wind-Num-2}
\end{figure}

\begin{theorem}\label{thm:glucktwist}
For each integer $w$, there exists a knot $K \subseteq S^1\times S^2$ with winding number $w$ such that $K$ is not isotopic to $\G(K)$.

Moreover, if $K\subseteq S^1\times S^2$ has odd winding number and is isotopic to $\G(K)$, then $K$ has geometric winding number $1$.\end{theorem} 

Observe that by the lightbulb trick (see e.g.~\cite[Exercise~9F4]{Rolfsen:1976-1}), there are only two isotopy classes of knots in $S^1\times S^2$ with geometric winding number 1.  The one with algebraic winding number $1$ is called the the \emph{Hopf knot} in~\cite{DNPR}.  The other is its reverse. 
 Since the only prime, reducible, oriented $3$-manifold is $S^1\times S^2$, Theorem~\ref{thm:glucktwist} and Theorem~\ref{thm:irred-rigid} together comprise the proof of Theorem~\ref{thm:main}.

For a knot $K\subseteq S^1\times S^2$, we may fix a diagram, such as in Figure~\ref{fig:Generic-Twist}. Such a diagram lies in $S^3$ and as such, has a preferred longitudinal framing. Then any framing of the diagram may be identified with an integer, by comparing with the preferred framing. The key observation for the proof of Theorem~\ref{thm:glucktwist} is that the Gluck twist and isotopies have differing effects on the framing of a given knot diagram. This is made precise in Proposition~\ref{prop:prelim}, which states that if $D$ is a diagram for a knot $K\subseteq S^1\times S^2$ with winding number $w$, such that $K$ is isotopic to $\G(K)$, then for any nonzero integer $f$, there is a 
diffeomorphism of $S^1\times S^2$ sending the $f$-framing of $D$ to the $(f+w^2+2kw)$-framing of $D$ for some integer $k$.  This observation along with a theorem of McCullough~\cite[Theorem~1]{McCullough06} regarding diffeomorphisms of $3$-manifolds restricting to nontrivial Dehn twists on the boundary implies that either $w^2+2kw=0$ or $K$ has geometric winding number $1$, proving the second half of Theorem~\ref{thm:glucktwist}.

For nonzero even winding numbers, the requirement that $w^2+2kw=0$ is not a contradiction. However, Proposition~\ref{prop:prelim} implies the existence of a non-trivial self-diffeomorphism of the manifold $M(D,f)$, for each $f$.  Here $M(D,f)$ is the result of Dehn surgery on $S^1\times S^2$ along $D$ with framing $f$.  For the knots obtained as $(2n,1)$ cables of the core curve $S^1\times \{pt\}\subseteq S^1\times S^2$, that is, as $(2n,1)$-cables of the Hopf knot, for $\lvert n\rvert>1$ (see Figure~\ref{fig:evenwinding}), we employ the Heegaard-Floer correction term defined by Ozsv\'{a}th and Szab\'{o}~\cite{Ozsvath-Szabo:2003-2} to obstruct the existence of such a diffeomorphism. For winding number $\pm 2$, we use SnapPy and Sage~\cite{SnapPy} to verify that for $D$ as in Figure~\ref{fig:windingtwo}, the manifold $M(D,1)$ is hyperbolic with no non-trivial self-diffeomorphisms. 

The arguments outlined in the preceding paragraphs are ineffectual when $w=0$.  In this case, we employ an approach using the satellite construction (Section~\ref{sec:zero}).  This idea was suggested to us by Charles Livingston after the preprint appeared on the arXiv.

\begin{remark}\label{rem:nonprime}
A key ingredient in our proof of Theorem~\ref{thm:freehomotopy}, and thereby Theorem~\ref{thm:main}, is that the map $\Phi\colon \Mod^+(M)\to \Out(\pi_1(M))$ is injective when $M$ is an irreducible $3$-manifold (Theorem~\ref{thm:mappingclass}). For general $M$, this is not the case, but the kernel of $\Phi$ is generated by Gluck twists along the essential spheres in $M$~\cite[Theorem~1.5]{mccullough:nonprime-mcg} (see also~\cite[Proposition~2.1]{hatcher-wahl:stabilization-mcg}). Laudenbach showed in~\cite{laudenbach:annals-1973,laudenbach:homotopie-et-isotopie} that if $S$ and $S'$ are homotopic $2$-spheres in a $3$-manifold, they are isotopic. So indeed the kernel of $\Phi$ is generated by the Gluck twists along (a) the connected sum spheres of $M$ and (b) the nonseparating spheres in each $S^1\times S^2$ summand of $M$. The case of nonseparating spheres could presumably be addressed as in the proof of Corollary~\ref{cor:nS1xS2}. Therefore it seems the main remaining challenge in trying to extend Theorem~\ref{thm:main} to reducible $3$-manifolds is the case of Gluck twists along separating $2$-spheres. In $\#^n S^1\times S^2$, such Gluck twists along separating $2$-spheres were shown to be isotopic to the identity by Laudenbach ~\cite{laudenbach:annals-1973,laudenbach:homotopie-et-isotopie}. But in general they are not even homotopic to the identity~\cite{hendriks:dehn-separating,friedman-witt:dehn-separating}.
\end{remark}

It is astonishing that Theorem~\ref{thm:glucktwist} can be proven by studying \textit{diffeomorphism} invariants of 3-manifolds, even though $K$ is sent to $\G(K)$ by a diffeomorphism. A more obvious approach to this problem is to begin with an isotopy invariant of knots which has no reason to be preserved under diffeomorphism.  One such invariant is the \emph{Kauffman bracket skein module} $\S(S^1\times S^2)$ associated to $S^1\times S^2$ \cite{Przytycki-1991}. This module is generated by 
unoriented framed links
in $S^1\times S^2$ and is explicitly computed in \cite{HoPr-1995}. Using their classification we prove that this invariant is unable to detect the difference between $K$ and $\G(K)$. Let $\K$ be a framed knot that corresponds to $K\subseteq S^1\times S^2$ when one forgets the framing. For an integer $f$, let $\K^f$ denote the framed knot obtained by adding $f$ full right-handed twists to the framing of $\K$. Observe that $\K$ and $\K^f$ have the same underlying knot type.

\begin{theorem}\label{thm:Gluck-and-skein}
Let $\K$ be a framed knot in $S^1\times S^2$ with geometric winding number $w$ and let $[\K]$ be its class in $\S(S^1\times S^2)$. If $w$ is even then $[\K] = [\G(\K)]$. If $w$ is odd, then $[\K]=[\G(\K)^f]$ for some $f\in \Z$.

In particular, the class in the Kauffman skein bracket module does not distinguish the knot type of $K$ from that of $\G(K)$.
\end{theorem}

\subsection*{Outline}
In Section~\ref{sec:propA} we prove Theorem~\ref{thm:PropertyA}. Section~\ref{sec:proof-irred} gives the proof of Theorem~\ref{thm:irred-rigid} using Theorem~\ref{thm:PropertyA}. In Section~\ref{sec:glucktwist} we establish Proposition~\ref{prop:prelim} and prove Theorem~\ref{thm:glucktwist} for odd winding numbers. The case of even, non-zero winding numbers are addressed in Section~\ref{sec:even}, and of zero winding number in Section~\ref{sec:zero}.
Theorem~\ref{thm:Gluck-and-skein} is proven in Appendix~\ref{appendixA}. 

\subsection*{Notation and conventions} We work in the smooth category, unless stated otherwise. Throughout this paper, all manifolds are assumed to be connected and either orientable or oriented. For any knot $K$ in $S^3$ and any integer $n \in \Z$, we denote by $S^3_n(K)$ the $3$-manifold obtained by $n$-framed Dehn surgery on $S^3$ along $K$. The symbol $\cong$ denotes either an isomorphism between groups or a diffeomorphism between manifolds.

\subsection*{Acknowledgments}
Much of this project was carried out while the authors were variously based at or visiting the Max Planck Institute for Mathematics and/or Georgia Tech. We thank both institutions for their hospitality and for providing an excellent environment for research. We also thank Stefan Friedl, Marco Golla, Tye Lidman, Lisa Piccirillo, Danny Ruberman, and Laura Starkston for helpful conversations. We thank Chuck Livingston for the idea of the proof of Theorem~\ref{thm:glucktwist} in the winding number zero case, as well as several helpful comments on the exposition. We are grateful to Yago Antol\'in for pointing out several references on Property A for acylindrically hyperbolic groups and conjugacy separability for 3-manifold groups. We are grateful to the anonymous referees for their thoughtful and detailed feedback.

The second author is supported by NSF grant DMS-1906269. The fourth author was partially supported by the Samsung Science and Technology Foundation (SSTF-BA2102-02) and by the NRF grant RS-2025-00542968.

\subsection*{Note} 
After hearing a talk on this paper, John Etnyre and Dan Margalit communicated to us that they have a different proof of Theorem~\ref{thm:main} which extends to the non-prime case.  Their proof is yet to appear. For the case of $\#^n S^1\times S^2$, see Corollary~\ref{cor:nS1xS2}.

\section{Property A of groups}\label{sec:propA}

The goal of this section is to prove Theorem~\ref{thm:PropertyA}, which we restate for the reader's convenience. 

\begin{reptheorem}{thm:PropertyA}Let $M$ be a closed, orientable $3$-manifold. Then $\pi_1(M)$ has Property~A.  
\end{reptheorem}
\begin{proof}
 Generalizing the work of Grossman~\cite{Gr75}, Minasyan and Osin showed that every torsion-free hyperbolic and relatively hyperbolic group has Property~A~\cite{MO10}. In particular, since a non-trivial free product is hyperbolic relative to its factors, every non-trivial free product has Property A. As mentioned in the introduction, this was originally proved by Neshchadim~\cite{Neshchadim:1996}. Since for $3$-manifolds $M, M'$, we have that $\pi_1(M\#M')\cong \pi_1(M)*\pi_1(M')$, this implies in particular that non-prime 3-manifold groups have Property A. Since $\Z=\pi_1(S^1\times S^2)$ has Property A, we may assume $M$ is irreducible.

Antol\'{i}n--Minasyan--Sisto \cite{AMS2016} generalized the results of~\cite{MO10} to the case of acylindrically hyperbolic groups. By definition, acylindrically hyperbolic groups admit an acylindrical action on a hyperbolic space.  From this action one can recover many hyperbolic--like properties of the group. In particular, they showed that acylindrically hyperbolic groups enjoy Property A.

By \cite[Corollary 2.9]{MO2015}, either $\pi_1(M)$ is acylindrically hyperbolic, Seifert-fibered or virtually polycyclic.  The class of acylindrically hyperbolic groups is disjoint from the other two classes and contains all non-prime 3-manifold groups as well as all those that have a non-trivial JSJ-decomposition but whose fundamental groups are not virtually polycyclic. If $\pi_1(M)$ is acylindrically hyperbolic then $\pi_1(M)$ has Property A by the result of Antol\'{i}n--Minasyan--Sisto \cite{AMS2016} quoted above. 

This leaves only two cases to check. The case of Seifert-fibered $M$ is addressed below in Theorem \ref{thm:SeifertA}. The virtually polycyclic but not Seifert-fibered case consists of those closed 3-manifolds supporting Sol-geometry.  Any such 3-manifold is either a torus bundle over the circle with Anosov monodromy or has a nontrivial double cover of that form (see e.g.\ \cite[Theorem 1.8.2]{Friedlbook}). We will refer to the former as an \emph{Anosov bundle} and the latter as an \emph{Anosov double}.  Anosov doubles are formed by gluing two twisted $I$-bundles over the Klein bottle together by a
 diffeomorphism along their common torus boundary. This case is addressed below in Proposition \ref{prop:AnosovA}. This completes the proof modulo the latter two quoted results, that is, Theorem \ref{thm:SeifertA} and Proposition \ref{prop:AnosovA}.
\end{proof}

First we address the case of Seifert fibered 3-manifolds.
Each Seifert-fibered 3-manifold fibers over a 2-dimensional orbifold.  In many cases, Property A is already known to hold:
\begin{theorem}[\cite{AKT2003, ATK:2009}]\label{thm:AKTseifert} If $G$ is a Seifert-fibered 3-manifold group whose base orbifold is not a sphere with precisely 3 cone points nor a torus with a single cone point, then $G$ has Property A.
\end{theorem}

It remains to prove Property A when the base orbifold is a sphere with 3 cone points or a torus with 1 cone point.  For each of these, since the base is orientable, the 3-manifold group will be a central extension of the 2-dimensional orbifold group.

We will write $S^2(p,q,r)$ to denote the $2$-sphere with precisely three cone points of order $p,q,r\geq 2$, and $T^2(n)$ to denote the torus with a single cone point of order $n\geq2$.  We will denote the corresponding groups by $G_0(p,q,r)$ and $G_1(n)$, respectively. Recall that both $G_1(n)$ and $G_0(p,q,r)$ for $p,q,r,n\geq 2$ are non-elementary discrete subgroups of $\Isom^+(\Hy^2)$ {if  $\frac{1}{p}+\frac{1}{q}+\frac{1}{r}<1$}, and thus satisfy the following:
\begin{lemma} Let $\Gamma\leq \Isom^+(\Hy^2)$ be a non-elementary discrete group of isometries. Then:
\begin{enumerate}
\item The centralizer of every nontrivial element of $\Gamma$ is cyclic.
\item $\Gamma$ has trivial center.
\item $\Gamma$ does not contain nontrivial finite normal subgroups.
\end{enumerate}
\end{lemma}
\begin{proof} Point (1) follows from the classification of isometries of $\Isom^+(\Hy^2)=\PSL_2(\R)$ into elliptic, parabolic, and hyperbolic type (see, e.g. \cite[Chapter 2]{Katok-Fuchsian}). Any two commuting elements must have the same fixed point set in $\smash{\overline{\Hy}^2=\Hy^2\cup \partial_\infty \Hy^2}$, which is either a point in the interior of $\Hy^2$ (elliptic), a point on the boundary at infinity $\partial_\infty\Hy^2$ (parabolic), or a pair of points in $\partial_\infty\Hy^2$ (hyperbolic). In $\Isom^+(\Hy^2)$ the corresponding stabilizers are isomorphic to $\SO(2)$ (elliptic), $\R$ (parabolic), and $\R$ (hyperbolic). Since $\Gamma$ is discrete, in each case the centralizer must be cyclic. Statement (2) then follows since $\Gamma$ was assumed to be non-elementary hence  not itself virtually cyclic.  Statement (3) also follows from the fact that $\Gamma$ is non-elementary and that if $N\unlhd \Gamma$ is a finite normal subgroup then the centralizer $C_\Gamma(N)$ has finite-index in $\Gamma$.
\end{proof}
We thus reduce the proof of Property A for Seifert-fibered groups to the proof of Property A for 2-dimensional orbifolds via the following key lemma. 
\begin{lemma} \label{lem:CentralExtension}
Let $\Gamma$ be a group with center $Z(\Gamma)$. Suppose $\Gamma/Z(\Gamma)$ has Property A, trivial center, and is generated by elements with cyclic centralizer.  Then $\Gamma$ has Property A.  
\end{lemma}
\begin{remark}
As a result of the above lemma, if $\Gamma$ is a group with center $Z(\Gamma)$ and $G\le Z(\Gamma)$, such that $\Gamma/G$ has Property A, trivial center, and is generated by elements with cyclic centralizer, then $\Gamma$ has Property A. This follows since if $G\le Z(\Gamma)$ and $\Gamma/G$ has trivial center then $G=Z(\Gamma)$.
\end{remark}

\begin{proof}[Proof of Lemma~\ref{lem:CentralExtension}] Let $Z=Z(\Gamma)$ and let $q:\Gamma\rightarrow \Gamma/Z$ be the quotient map. Choose elements $\{x_i\}_{i\in I}\subseteq \Gamma$, where $I$ is an indexing set, such that $\{q(x_i)\}_{i\in I}$ generates $\Gamma/Z$ and each $q(x_i)$ has cyclic centralizer.

Suppose $\varphi:\Gamma\rightarrow \Gamma$ is class preserving. Clearly $\varphi|_Z=\text{Id}$.  Thus, $\varphi$ induces a class preserving automorphism $\overl{{\varphi}}:\Gamma/Z\rightarrow \Gamma/Z$.  Since $\Gamma/Z$ has Property A, $\overl{{\varphi}}$ is inner, i.e. $\overl{{\varphi}}=c_{\bar{g}}$ is conjugation by some $\bar{g}\in \Gamma/Z$.  As $\Gamma/Z$ has trivial center, $\Inn(\Gamma/Z)\cong \Gamma/Z\cong \Inn(\Gamma)$.  Choose any $g\in \Gamma$ such that $q(g)=\bar{g}$, and consider $\psi=c_g^{-1}\circ \varphi$. We know that $\psi$ preserves every coset $xZ$, $x\in \Gamma$, and acts by conjugation on $x$ for each $x\in \Gamma$. 
 
Fix $i\in I$. Then we have $\psi(x_i)=yx_iy^{-1}=x_iz$ for some $z\in Z$ and some $y\in \Gamma$.  Then $[y,x_i]\in Z$, hence $q(y)$, $q(x_i)$ commute.  Because the centralizer of $q(x_i)$ is cyclic, $q(x_i)$ and $q(y)$ are powers of the same element.  It follows that $y$ is in the centralizer of $x_i$ in $\Gamma$.  Thus, $yx_iy^{-1}=x_i$, and $\psi|_{x_iZ}=\text{Id}$.
Since the $q(x_i)$ generate $\Gamma/Z$, we conclude that $\psi=\text{Id}$, and that $\varphi$ is inner.
\end{proof}

Recall that 2-dimensional orbifolds are divided into 3 classes: spherical, Euclidean and hyperbolic.   For every $n\geq 2$, $T^2(n)$ is hyperbolic, but for the 2-sphere with 3 cone points we have the following trichotomy (see~\cite{Scott83})
\[\textrm{$S^2(p,q,r)$ is } \left\{\begin{array}{c}\textrm{spherical}\\ \textrm{Euclidean} \\\textrm{hyperbolic}\end{array}\right\} \textrm{ according as } \frac{1}{p}+\frac{1}{q}+\frac{1}{r} \textrm{ is }\left\{\begin{array}{c}>\\ =\\<\end{array}\right\} 1\]

For the spherical case, there is one infinite family and 3 exceptional cases, while for the Euclidean case there are only 3 cases. Everything else is hyperbolic:
\begin{itemize}
\item Spherical: $(2,2,n)$, $n\geq 2$; $(2,3,3)$; $(2,3,4)$; $(2,3, 5)$;
\item Euclidean: $(2,3,6)$; $(2,4,4)$; $(3,3,3)$;
\item hyperbolic: all others.
\end{itemize}

In order to deal with the hyperbolic orbifolds, we will rely on a result of Minasyan--Osin \cite{MO10}, concerning normal automorphisms of (relatively) hyperbolic groups.  Recall that an automorphism $\phi$ of a group $G$ is \emph{normal} if $\phi(N)=N$ for every $N\trianglelefteq G$. Observe that every class preserving automorphism of $G$ is normal since any normal subgroup of a group is a union of conjugacy classes.
For (relatively) hyperbolic groups, we have the following strengthening of Property A.  

\begin{theorem}[{\cite[Corollary 1.2(b)]{MO10}}]\label{thm:MOnormal}Suppose $G$ is a $($relatively$)$ hyperbolic group which is non-cyclic and which does not contain any non-trivial finite normal subgroups. Then every normal automorphism of $G$ is inner.  
\end{theorem}

Note that this result is subsumed by the results of \cite{AMS2016}, but we only need (relative) hyperbolicity here. Using this theorem and Lemma \ref{lem:CentralExtension} we deduce 
\begin{proposition}\label{prop:HypOrb} If the base orbifold for a Seifert-fibered $M$ is hyperbolic, then $\pi_1(M)$ has Property A. 
\end{proposition}
\begin{proof}By Theorem \ref{thm:AKTseifert}, the only remaining cases are when $M$ fibers over $T^2(n)$, $n\geq 2$ or $S^2(p,q,r)$ with $\frac{1}{p}+\frac{1}{q}+\frac{1}{r}<1$. Each of these base spaces is orientable, hence $\pi_1(M)$ is a central extension of either $G_1(n)$ or $G_0(p,q,r)$ by $\Z$.

 As noted earlier, both $G_1(n)$ and $G_0(p,q,r)$ are nonelementary, discrete subgroups of $\Isom^+(\Hy^2)$, and so have trivial center, no finite normal subgroups and the centralizer of every nontrivial element is cyclic. By Theorem \ref{thm:MOnormal}, both $G_1(n)$ and $G_0(p,q,r)$ have Property A.  Hence, $\pi_1(M)$ verifies the hypotheses of Lemma \ref{lem:CentralExtension},  so we conclude that $\pi_1(M)$ has Property A, as desired.
\end{proof}

It remains to prove Property A for the spherical and Euclidean cases, each of which corresponds to a sphere with 3 cone points. We refer the reader to \cite[Chapter II.4]{Mag74} for the basic facts about triangle groups and their index two counterparts, von Dyck groups, that we will use here.  The latter are the orbifold fundamental groups of spheres with 3 cone points. The presentation of such a group is 
\begin{equation}\label{eqn:pres of G0}G_0(p,q,r)=\langle x,y,z \mid x^p=y^q=z^r=xyz=1\rangle.\end{equation}

\begin{proposition}\label{prop:EucOrb} If the base orbifold for a Seifert-fibered $M$ is Euclidean, then $\pi_1(M)$ has Property A.
\end{proposition}
\begin{proof} By Theorem \ref{thm:AKTseifert}, the only remaining case is when $M$ fibers over a Euclidean $S^2(p,q,r)$. Since $S^2(p,q,r)$ is orientable, $\pi_1(M)$ is a central extension of $G_0(p,q,r)$ by $\Z$. Further, the group $G_0(p,q,r)$ is an index 2 subgroup of the corresponding triangle group, and acts faithfully on the tessellation of $\R^2$ by triangles with angles $(\frac{\pi}{p},\frac{\pi}{q},\frac{\pi}{r})$.  We will exploit this action to verify the hypotheses of Lemma \ref{lem:CentralExtension}. 
The generators $x$, $y$, and $z$ appearing in the presentation \eqref{eqn:pres of G0} act by counterclockwise rotations by $\frac{2\pi}{p}$, $\frac{2\pi}{q}$, and $\frac{2\pi}{r}$ respectively about the three vertices of a fixed triangle in the tessellation.
  The sphere $S^2(p,q,r)$ is the union of two triangles in the tessellation with angles $(\frac{\pi}{p},\frac{\pi}{q},\frac{\pi}{r})$ and thus a fundamental domain for the action is the union of any two triangles that share an edge. 

$G_0(p,q,r)$ is a discrete subgroup of $\Isom^+(\R^2)=\R^2\rtimes \SO(2)$,  and every finite subgroup 
 therefore injects under the projection $G_0(p,q,r)\rightarrow \SO(2)$.  Each finite subgroup is therefore cyclic and 
stabilizes some point in $\R^2$.
We claim that this point is a vertex of the tessellation.  Indeed, if this point were on an edge of the tessellation, then each element $\sigma\in G$ would permute the vertices at the ends of this edge. Since the orbits of the vertices of a triangle in the tessellation are disjoint, $\sigma$ fixes this edge pointwise.  As $\sigma$ is an orientation-preserving isometry of the plane, $\sigma$ is the identity.  A similar argument applies if the fixed point in interior to a triangle.
 The fact that there are 3 orbits of vertices now implies that there are exactly 3 conjugacy classes of maximal, finite cyclic subgroups.

If we identify $\R^2$ with $\CC$, in each case the vertices will be the points of a lattice $\Lambda\subseteq \CC$, and each element of $G_0(p,q,r)$ can be thought of as an affine transformation of the form $z\mapsto az+b$, where $a, b\in \CC$ and $a\overline{a}=1$. If $g\colon z\mapsto az+b$ is nontrivial and has nonempty fixed set, then $a\neq 1$ and hence {$\text{Fix}(g)=b/(1-a)\in \CC$} is the unique fixed point. It follows that if two nontrivial elements $g,h$ commute then either they are both translations, hence have no fixed points, or they have the same fixed point. Indeed, if $\text{Fix}(h)=z_0$ then $g(z_0)=g(h(z_0))=h(g(z_0))$, so $g(z_0)$ is fixed by $h$ and therefore $g(z_0)=z_0$.  We conclude that $G_0(p,q,r)$ has trivial center and that the centralizers of the the three generators are cyclic. Thus, $\pi_1(M)$ satisfies two of the hypotheses of Lemma \ref{lem:CentralExtension}, hence to prove $\pi_1(M)$ has Property A it suffices to show that $G_0(p,q,r)$ has Property A.\\

\begin{figure}[htb]
\includegraphics[width=.5\textwidth]{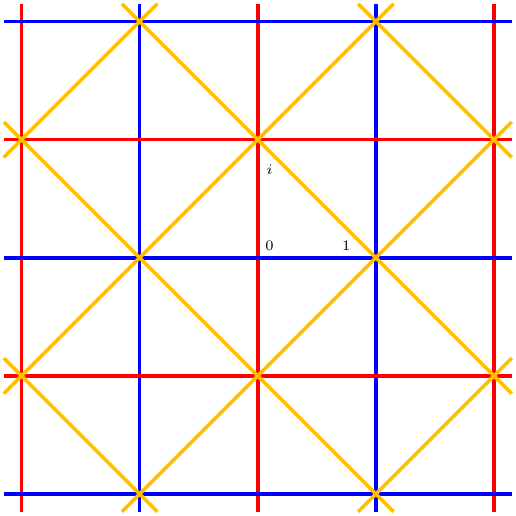}
\caption{The tessellation of $\R^2$ by $(\frac{\pi}{2},\frac{\pi}{4},\frac{\pi}{4})$-triangles. }\label{fig:244}
\end{figure}

\noindent
\textbf{(2,4,4):} In this case $\Lambda=\Z[i]$, and we choose our base triangle to be the one with vertices $0$, $1$, $i$. See Figure \ref{fig:244}. The generators are then given by rotations about $0$, $1$, and $i$, respectively. They have the form 
\begin{eqnarray*}T_0&\colon& z\mapsto -z\\
T_1&\colon& z\mapsto i(z-1)+1=iz+(1-i)\\
T_i&\colon& z\mapsto i(z-i)+i=iz+(1+i)
\end{eqnarray*}

Suppose $\phi:G_0(2,4,4)\rightarrow G_0(2,4,4)$ is class preserving.  After post-composing with an inner automorphism, we may assume that $\phi(T_0)=T_0$.  Since $\phi(T_1)$ is conjugate to $T_1$, $\phi(T_1)$ is a counterclockwise rotation by $\frac{\pi}{2}$ at another vertex in the orbit of the point $1$. The orbit of 1 is all of the points with $x$ coordinate odd and $y$ coordinate even, i.e., all of the points of the form $(2a+1) +(2b)i$ for some $a,b\in \Z$.
We can thus assume that \[\phi(T_1)\colon z \mapsto iz+(2a+2b+1)+ (2b-2a-1)i.\]
Similarly, $\phi(T_i)$ is a counterclockwise rotation by $\frac{\pi}{2}$ at another vertex in the orbit of the point $i$, which are all the points of the form $(2c)+(2d+1)i$ for some $c,d\in\Z$. Hence \[\phi(T_i)\colon z\mapsto iz+(2c+2d+1)+ (2d-2c+1)i.\]
The fact that $\phi$ is an automorphism
 means that \[\phi(T_0)\circ\phi(T_1)\circ\phi(T_i)=T_0\circ\phi(T_1)\circ\phi(T_i)=\Id.\]  Since $T_0^{-1}=T_0$, we therefore must have 
\begin{eqnarray*}-z&=&\phi(T_1)\circ\phi(T_i)(z)\\
&=&\phi(T_1)(iz+(2c+2d+1)+ (2d-2c+1)i)\\
&=& i(iz+(2c+2d+1)+ (2d-2c+1)i)+(2a+2b+1)+ (2b-2a-1)i\\
&=& -z+(2a+2b+1-2d+2c-1)+ (2c+2d+1+2b-2a-1)i\\
&=& -z+ 2(a+b-d+c) +2i(c+d+b-a)
\end{eqnarray*}
We thus obtain that \begin{eqnarray*}
a+b-d+c&=&0\\
c+d+b-a &=&0
\end{eqnarray*}
It follows that $d=a$ and $c=-b$. Substituting, we see that $\phi(T_i)$ is a rotation about the point $(-2b,2a+1)$. The points $(0,0)$, $(2a+1,2b)$, $(-2b,2a+1)$ are the vertices of an isosceles right triangle which is one half of a fundamental domain for the action of $\langle T_0,~\phi(T_1),~ \phi(T_i)\rangle$ on $\R^2$.
The ratio of the area of this triangle to the one with vertices 0, 1, and $i$ is the index of $\langle T_0,~\phi(T_1),~ \phi(T_i)\rangle$. Since $T_0,~\phi(T_1),~ \phi(T_i)$ generate the whole group, that triangle has area $\tfrac{1}{2}$. This is only possible if
$a=0,~b=0$ or $a=-1,~b=0$. In the first case, $\phi$ is the identity.  In the second case, $\phi$ is conjugation by $T_0$. Hence $\phi$ is inner, as desired.  \\

\begin{figure}[htb]
\includegraphics[width=.5\textwidth]{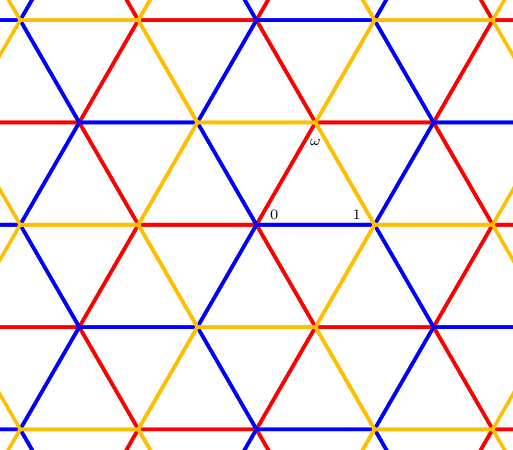}
\caption{The tessellation of $\R^2$ by $(\frac{\pi}{3},\frac{\pi}{3},\frac{\pi}{3})$-triangles. }\label{fig:333}
\end{figure}

\noindent
\textbf{(3,3,3):} In this case $\Lambda=\Z[\omega]$, where $\omega=e^{\frac{\pi i}{3}}$ and we choose our base triangle to be the one with vertices $0$, $1$, $\omega$. See Figure~\ref{fig:333}. Recall that $\omega$ satisfies $\omega^2-\omega+1=0$ and thus $\omega^3=-1$. The generators are 
\begin{eqnarray*}T_0&\colon& z\mapsto \omega^2z\\
T_1&\colon& z\mapsto \omega^2(z-1)+1=\omega^2z+(1-\omega^2)=\omega^2z+ (2-\omega)\\
T_\omega&\colon& z\mapsto \omega^2(z-\omega)+\omega=\omega^2z+(1+\omega)
\end{eqnarray*}
 
Suppose $\phi$ is class preserving.  After post-composing with an inner automorphism, we may assume that $\phi(T_0)=T_0$.  The lattice points in the orbit of $1$ are all the points of the form $a+b\omega$,
such that $a+2b\equiv 1\mod 3$, and the lattice points in the orbit of $\omega$ are  all the points of the form $c+d\omega$, such that $c+2d\equiv 2 \mod 3$.  Since $\phi(T_1)$ is conjugate to $T_1$, $\phi(T_1)$ is a counterclockwise rotation by $\frac{2\pi}{3}$ at another vertex in the orbit of the point $1$. We can thus assume that \[\phi(T_1)\colon z \mapsto \omega^2z +(2a+b) +(b-a)\omega.\]
Similarly, $\phi(T_\omega)$ is a counterclockwise rotation by $\frac{2\pi}{3}$ at another vertex in the orbit of the point $\omega$, hence has the form  \[\phi(T_\omega)\colon  z \mapsto \omega^2z +(2c+d) +(d-c)\omega\]
The fact that $\phi$ is an automorphism means that \[\phi(T_0)\circ\phi(T_1)\circ\phi(T_\omega)=T_0\circ\phi(T_1)\circ\phi(T_\omega)=\Id.\] Since $T_0^{-1}(z)=\omega^4z$, we must have an equality 

\begin{eqnarray*}\omega^4z&=&\phi(T_1)\circ\phi(T_\omega)(z)\\
&=&\phi(T_1)(\omega^2z +(2c+d) +(d-c)\omega)\\
&=& \omega^2(\omega^2z +(2c+d) +(d-c)\omega)+(2a+b) +(b-a)\omega \\
&=& \omega^4z + \omega^2(2c+d) +(d-c)\omega^3+ (2a+b) +(b-a)\omega\\
&=& \omega^4z + (\omega-1)(2c+d)- (d-c)+ (2a+b) +(b-a)\omega\\
&=& \omega^4z + (2a+b-c-2d) +(2c+d+b-a)\omega
\end{eqnarray*}
We thus obtain that \begin{eqnarray*}
2a+b-c-2d&=&0\\
2c+d+b-a &=&0
\end{eqnarray*}
It follows that $d=a+b$ and $c=-b$. The points $0$, $a+b\omega$, $-b+(a+b)\omega$ form the vertices of an equilateral triangle which is one half the fundamental domain for the action of $T_0,\phi(T_1), \phi(T_\omega)$ on $\R^2$.  The only way this triangle can have area equal to $\frac{\sqrt{3}}{4}$ is if $(a,b) \in \{(1,0),(-1,1),(0,-1) \}$. But then $\phi$ is the identity, conjugation by $T_0$, or conjugation by $T_0^2$, respectively. \\ 

\begin{figure}[htb]
\includegraphics[width=.5\textwidth]{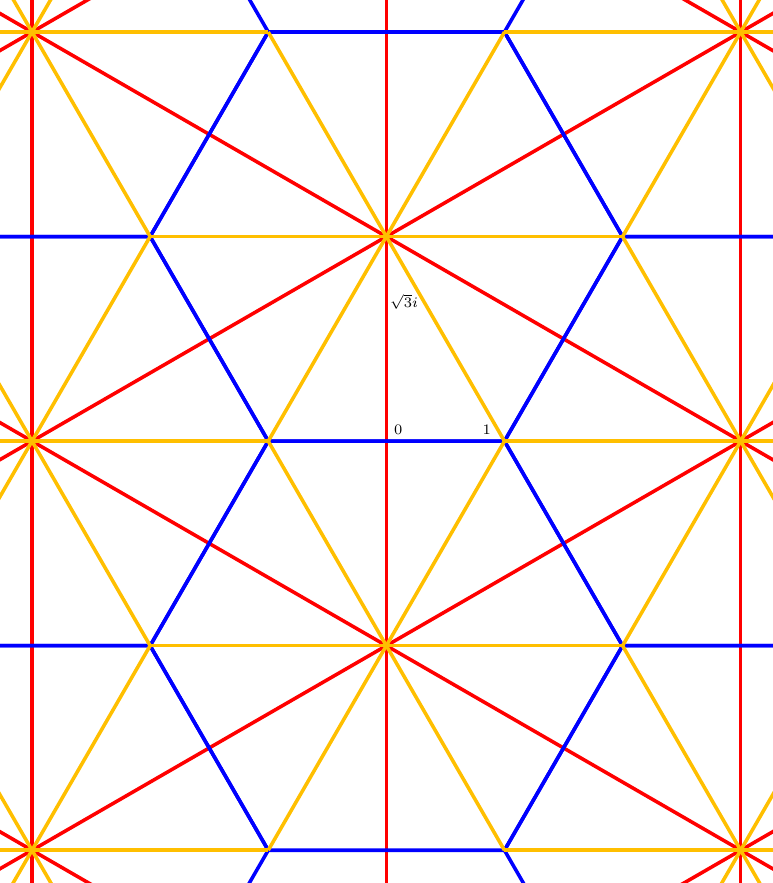}
\caption{The tessellation of $\R^2$ by $(\frac{\pi}{2},\frac{\pi}{3},\frac{\pi}{6})$-triangles. }\label{fig:236}
\end{figure}

\noindent
\textbf{(2,3,6):} In this case the vertices of the triangle are at $0$, $1$ and $i\sqrt{3}$, as in Figure~\ref{fig:236}. The generators are the rotation $T_0$ by $\pi$ at $0$, the rotation $T_1$ by $\frac{2\pi}{3}$ at $1$ and the rotation $T_{i\sqrt3}$ by $\frac{\pi}{3}$ at $i\sqrt{3}$, all counterclockwise. 

As in the previous cases, given a class preserving automorphism $\phi$, we can assume after postcomposing with an inner automorphism that $\phi(T_0)=T_0$. Then $\phi(T_1),\phi(T_{i\sqrt 3})$ are rotations about lattice points in the orbit of $1$ and $i\sqrt{3}$, respectively. We let $\smash{\xi=e^{\frac{\pi i}{6}}}$, a $12^{th}$ root of unity. Then $\xi^2=\omega$ is a $6^{th}$ root of unity.  The points in the orbit of $1$ are all of the form $ \pm 1+2ai\sqrt{3} +2b\xi\sqrt{3}$, while those in the orbit of $i\sqrt 3$ are of the form $(1+2c)i\sqrt3+2d\xi\sqrt{3}$, for some $a,b,c,d\in \Z$. Using the relations $i\sqrt{3}=2\omega-1$ and $\xi\sqrt{3}=\omega+1$ we rewrite these to get that a point in the orbit of 1 has the form $(\pm1-2a+2b)+(4a+2b)\omega$ and a point in the orbit of $i\sqrt3$ has the form $(-1-2c+2d)+(2+4c+2d)\omega$. We first consider the case where $\phi(T_1)$ is a rotation about a point of the form $(1-2a+2b)+(4a+2b)\omega$.  As above, since $T_0^{-1}=T_0$, we have
\begin{eqnarray*}
-z &=& \phi(T_1) \circ \phi(T_{i\sqrt{3}})(z) \\
&=& \phi(T_1)\Big( \omega(z - (-1-2c+2d)-(2+4c+2d)\omega) + (-1-2c+2d)+(2+4c+2d)\omega \Big) \\
&=& \phi(T_1)\Big(\omega z+(1+2c+4d)+(1+2c-2d)\omega\Big)\\
&=& \omega^3z + (1 + 2c+4d)\omega^2+(1+2c-2d)\omega^3-(1-2a+2b)\omega^2 -(4a+2b)\omega^3\\
&&\quad +~ (1-2a+2b)+(4a+2b)\omega\\
&=& -z + (6b-4c-2d)+(6a+2c+4d)\omega
\end{eqnarray*}

Therefore we must have that $3b-2c-d=0$ and $3a+c+2d=0$. Solving for $d$ and $c$ we obtain $c=a+2b$ and $d=-2a-b$. Therefore, $T_0$, $\phi(T_1)$ and $\phi(T_{i\sqrt3})$ are rotations by $\pi$, $\frac{\pi}{3}$, and $\frac{\pi}{6}$, respectively, about the corners of a triangle with vertices at $0$, $(1-2a+2b)+ (4a+2b)\omega$, and $(-1-6a-6b)+(2+6b)\omega$. Twice the area of this triangle (and hence the area of a fundamental domain for the action) is \[(1+3b)^2\sqrt{3}+(b+2a)^23\sqrt3,\]  which is only equal to $\sqrt3$ when $a=b=0$.
We conclude that $\phi(T_0)=T_0$, $\phi(T_1)=T_1$ and $\phi(T_{i\sqrt3})=T_{i\sqrt3}$ so $\phi=\Id$, as desired. The other case, where $\phi(T_1)$ has the form $(-1-2+2b)a+(4a+2b)\omega$, is similar but the conclusion is that $\phi$ is conjugation by $T_0$, hence inner. This concludes the proof for all 3 cases, and the proposition follows.
\end{proof}

Next, we turn to the spherical case.

\begin{proposition}\label{prop:SphereOrb}If the base orbifold for a Seifert-fibered $M$ is spherical, then $\pi_1(M)$ has Property~A.
\end{proposition}
\begin{proof}If the base orbifold is spherical, then $M$ supports $\Sa^3$-geometry. Thus $\pi_1(M)$ is a finite subgroup of $\SO(4)$ which acts freely on $\Sa^3$. These are all of the form $G\times \Z/k$ where $k$ is coprime to $|G|$, and $G$ is either trivial or on the following list (see \cite[Section~1.7]{Friedlbook}):
\begin{enumerate}
\item $Q_{8n}=\langle x, y~|~x^2=(xy)^2=y^{2n}\rangle,$ $n\geq 1$

\item $P_{48}$, $P_{120}$

\item $D_{2^m(2n+1)}=\langle x, y~|~x^{2^m}=y^{2n+1}=1, ~xyx^{-1}=y^{-1}\rangle,$ $m\geq 2,$ $n\geq1$

\item $P'_{8\cdot3^m}=\langle x, y,z ~|~x^2=(xy)^2=y^2, ~zxz^{-1}=y, ~zyz^{-1}=xy,~ z^{3^m}=1\rangle,$ $m\geq 1$

\end{enumerate}

The groups $P_{48}$ and $P_{120}$ correspond to the orbifolds $S^2(2,3,4)$ and $S^2(2,3,5)$, respectively.  The groups are central $\Z/2$-extensions of $S_4$ and $A_5$.  The  outer automorphism group of $S_4$ is trivial and the outer automorphism group of $A_5$ is $\Z/2$ generated by conjugation by a transposition. There are two conjugacy classes of 5-cycles in $A_5$ but only one in $S_5$.  The non-trivial outer automorphism of $A_5$ exchanges these two, hence is not class preserving.  Thus, both $S_4$ and $A_5$ have Property A. For $n\geq 3$, $A_n$ is generated by 3-cycles, and when $n\leq 5$, each of these has cyclic centralizer.  To generate all of $S_4$, we add a 4-cycle, which also has cyclic centralizer.  Since both $S_4$ and $A_5$ have trivial center, Lemma \ref{lem:CentralExtension} applies, so we conclude that $P_{48}$ and $P_{120}$ also have Property A.

Recall that $A_4$ admits the presentation 
$\langle x, z\mid x^2=z^3=(xz)^3=1\rangle$. From this presentation,
  one can see that the subgroup generated by $z^3$ and $x^2$ is central, and therefore normal. Thus, there is a short exact sequence \[1\rightarrow \left\langle z^3 , x^2\right\rangle\rightarrow P'_{8\cdot3^m}\rightarrow A_4\rightarrow 1.\]
Since the center of $A_4$ is trivial, it is easy to verify that $\langle z^3,x^2\rangle$ is the center of $P'_{8\cdot3^m}$. The outer automorphism group of $A_4$ is $\Z/2$ generated by conjugation by a transposition in $S_4$.  There are two conjugacy classes of 3-cycles in $A_4$, which are exchanged by this automorphism. Hence this automorphism is not class preserving and $A_4$ has Property A. As in the $A_5$ case, $A_4$ is generated by 3-cycles, each of which have cyclic centralizer. We apply Lemma \ref{lem:CentralExtension} again to obtain that $P'_{8\cdot3^m}$ has Property A.

For any $m\geq 2, n\geq 1$, the center of  $D_{2^m(2n+1)}$ is generated by $x^2$. We prove the same for $Q_{8n}$.   It follows from the relation $x^2=y^{2n}$ that $x^2$ commutes with both generators $x$ and $y$, and so it is central. As observed in \cite[Section~1.7]{Friedlbook}, $Q_{8n}$ is an extension of $D_{4n}$ by $\Z/2$, $1\to\Z/2\to Q_{8n}\to D_{4n}\to 1$ where $Q_{8n}\to D_{4n}$ has kernel $\langle x^2\rangle$ and is defined by sending $x$ to the reflection and $y$ to the rotation, $\rho$, of order $2n$.  As the center of $D_{4n}$ is $\{1,\rho^{n}\}$, it follows that the center of $Q_{8n}$ is a subgroup of $\langle x^2, y^n\rangle$.  We need only show that $y^n$ is not central. We can use the relation $x^2=(xy)^2$ to see that $xy^nx^{-1} = y^{-n}$. Since $y^{2n}=x^2\neq 1$ $y$ is not order $2n$ and so cannot be central.

We now prove Property $A$ for $Q_{8n}$ and $D_{2^m(2n+1)}$ directly.
 From the relation $x^2=xyxy$ we obtain $xyx^{-1}=y^{-1}$ for $Q_{8n}$, while this relation is explicit in the presentation of $D_{2^m(2n+1)}$ above. Hence in each case the conjugacy classes for nontrivial powers of $y$ are exactly $\{\{y^i, y^{4n-i}\}\mid ~1\leq i\leq 2n\}$ for $Q_{8n}$ and $\{\{y^j, y^{2n+1-j}\}\mid~1\leq j\leq n\}$ for $D_{2^m(2n+1)}$.  Let $\phi:Q_{8n}\rightarrow Q_{8n}$ be class preserving. By composing $\phi$ with a conjugation we may assume $\phi(x)=x$. Then $\phi(y)=y$ or $\phi(y)=y^{-1}$; either way $\phi$ is inner.  The case for $D_{2^m(2n+1)}$ is completely analogous. 

Finally, since abelian groups have Property A and Property A is closed under taking products, we conclude that $\pi_1(M)= G\times \Z/k$ has Property A, where $G$ is one of the groups on the list above.  This completes the proof.
\end{proof}

Combining these results we obtain:
\begin{theorem}\label{thm:SeifertA}Let $\Gamma$ be the fundamental group of an orientable Seifert-fibered  $3$-manifold.
Then $\Gamma$ has Property A. 
\end{theorem}

\begin{proof} Let $M$ be a Seifert-fibered manifold. Then $M$ fibers over a hyperbolic, Euclidean, or spherical orbifold.  Thus $\pi_1(M)$ has Property A by one of Proposition \ref{prop:HypOrb}, \ref{prop:EucOrb}, or \ref{prop:SphereOrb}.  
\end{proof}

We now prove Property A for $M$ supporting Sol-geometry. Recall that in this case $M$ is either an Anosov bundle or an Anosov double.

\begin{proposition}\label{prop:AnosovA} Let $M$ be an Anosov bundle or an Anosov double.   Then $\pi_1(M)$ has Property A. 
\end{proposition}
\begin{proof}Let $\Gamma:=\pi_1(M)$. If $M$ is an Anosov bundle then $\Gamma$ fits into a short exact sequence\[1\rightarrow \Z^2\rightarrow \Gamma\rightarrow\Z\rightarrow 1.\] Since $\Z$ is free the sequence splits and we can write $\Gamma=\Z^2\rtimes_A \Z$, where the $\Z$-factor conjugates $\Z^2$ by an Anosov matrix $A\in \text{SL}_2(\Z)$. 

Let $x,y$ be the standard basis elements generating $\Z^2$ and $t$ be the standard generator of $\Z$. Then $x,y$ and $t$ together generate $\Gamma$ and we have $tvt^{-1}=Av$ for all $v\in \Z^2$.  First observe that for any $v, w\in \Z^2$, $v$ is conjugate to $w$ in $\Gamma$ if and only if for some $k\in \Z$ we have $A^k(v)=w$. Indeed, we have that $t^kvt^{-k}=A^k(v)$. Conversely, suppose that $w=gvg^{-1}$ for some $g\in \Gamma$. Since $\Gamma$ is a semi-direct product, we can write $g=at^m$ for some $a\in \Z^2$ and $m\in \Z$. Then 
\[w=gvg^{-1}=a(t^mvt^{-m})a^{-1}=a(A^m(v))a^{-1}=a+A^m(v)-a=A^m(v),\]
where we have used the fact that $a$ and $A^m(v)$ commute since both lie in $\Z^2$.

Let $\phi$ be a class preserving automorphism of $\Gamma$. Post-composing by conjugation by some element of $\Gamma$ we may assume $\phi(t)=t$.  Since $\Z^2$ is normal, we know that $\phi|_{\Z^2}$ can be represented by some matrix $P\in \text{GL}_2(\Z)$. \textit{A priori} $P$ sends $x\mapsto A^kx$ for some $k\in \Z$, and by post-composing $\phi$ with conjugation by $t^{-k}$, we can arrange that $\phi(x)=x$ as well. Note that this does not change the fact that $\phi(t)=t$. Observe that $A$ commutes with $P$ since \[P(Av)=\phi(Av)=\phi(tvt^{-1})=t(Pv)t^{-1}=A(Pv)\] for each $v\in\Z^2$. In particular, we see that $P(Ax)=Ax$. In other words, $P$ fixes both $x$ and $Ax$. Since $A$ is Anosov, it has two distinct and irrational eigenvectors.  As $x$ is integral, $Ax$ cannot be a multiple of $x$.
We now have that $P$ fixes the two linearly independent vectors $x, A(x)\in \Z^2$. This implies that $P=I$ and that $\phi=\Id$, proving the proposition for Anosov bundles.

Suppose now $M$ is an Anosov double with gluing map represented by \[A=\left(\begin{array}{cc}a&b\\c&d \end{array}\right)\in \text{SL}_2(\Z).\]  By van Kampen's theorem, we get a presentation
\[\Gamma=\left\langle t_1,x_1,t_2,x_2\mid t_1x_1t_1^{-1}=x_1^{-1},~t_2x_2t_2^{-1}=x_2^{-1},~t_2^2=(t_1^2)^ax_1^c,~x_2=(t_1^2)^bx_1^d\right\rangle\]
The first two relations correspond to the two twisted $I$-bundles, each of whose fundamental group is  that of the Klein bottle.  The boundary tori are generated by $\langle t_1^2,x_1\rangle$ and $\langle t_2^2,x_2\rangle$, which are then identified by $A$, giving the last two relations. Since $\Z^2=\langle t_1^2,x_1\rangle=\langle t_2^2,x_2\rangle$ is normal, 
$\Gamma$ fits into a short exact sequence:
\[1\rightarrow \Z^2\rightarrow \Gamma\rightarrow\Z/2*\Z/2\rightarrow 1\]
where the $t_i$, $i=1,2$ map to the two generators of the $\Z/2$ factors. As the quotient is a free product, in order to determine the conjugation action on the kernel, it is enough to know how $t_1$ and $t_2$ act. Let $T_i$ be the involution of $\Z^2$ induced by conjugation by $t_i$. By definition, the fact that $M$ is an Anosov double means that  $T_1T_2$ is Anosov. In particular, since $T_1T_2$ has infinite order but $T_1$ and $T_2$ have order two, we see that $T_1$ and $T_2$ do not commute.  If we take $\{t_1^2,x_1\}$ as a basis for $\Z^2$, then  $T_1$ has matrix representation \[B:=\left(\begin{array}{cc}1&0\\0&-1 \end{array}\right)\]
while $T_2$ has matrix representation $A^{-1}BA$, hence $BA^{-1}BA$ is an Anosov matrix. 

Let $\phi$ be a class preserving automorphism of $\Gamma$.  Since $\Z^2$ is normal, $\phi(\Z^2)=\Z^2$. Thus, $\phi$ induces a class preserving automorphism $\overl{\phi}$ of $\Z/2*\Z/2$.  The latter, being a non-trivial free product, has Property A by the result of~\cite{Neshchadim:1996}. After conjugating $\Gamma$ by an alternating word in $\{t_1, t_2\}$, possibly followed by a power of $x_1$, we may assume that $\overline{\phi} = \Id$, $\phi(t_1) = t_1$, and $\phi(t_2) = t_2 w$ for some $w \in \mathbb{Z}^2$. Now, the restriction $\phi|_{\mathbb{Z}^2}$ is an automorphism $S$ of $\Z^2$, which must satisfy
\begin{align*}
S(T_1(v)) &= \phi(t_1 v t_1^{-1}) = t_1 \phi(v) t_1^{-1} = T_1(S(v)), \\
S(T_2(v)) &= \phi(t_2 v t_2^{-1}) = t_2 w \phi(v) w^{-1} t_2^{-1} = T_2(S(v)),
\end{align*}
for all $v \in \mathbb{Z}^2$, since $w \in \mathbb{Z}^2$ commutes with $\phi(v)$.
  In particular, $S$ is in the centralizer of $T_1$ and $T_2$. A straightforward calculation shows that the centralizer $C(T_1)$ of $T_1$ is given by the four diagonal matrices:
\[\left(\begin{array}{cc}\pm1&0\\0&\pm1 \end{array}\right)=\{\pm I,\pm T_1\}.\]
Thus, $C(T_2)=A^{-1}C(T_1)A =\{\pm I, \pm T_2\}$. 
Since $T_1$ and $T_2$ do not commute, the intersection of the centralizers of $T_1$ and $T_2$ is just $\{\pm I\}$.  The fact that $\phi(t_1)=t_1$ implies that $\phi(t_1^2)=t_1^2$.  Hence, $S$ must be the identity so $\phi(x_1)=x_1$, $\phi(x_2)=x_2$, and $\phi(t_2^2)=t_2^2$.

On the other hand, if $t_2$ is conjugate to $\phi(t_2)=t_2w$ then there exists $v\in \Z^2$ such that $t_2w=vt_2v^{-1}=t_2((T_2-I)v)$, where we have used that $T_2=T_2^{-1}$. The image of $(T_2-I)$ is $\langle x_2^2\rangle$ so  $w\in \langle x_2^2\rangle$. Observe that conjugating by $t_1^2$ is the identity on  $t_1$ and $x_1$, and sends  $t_2$ to $t_2((T_2-I)t_1^2)=t_2((T_2-I)A^{-1}t_2^2)$. With respect to the basis $\{t_2^2, x_2\}$, we can represent $(T_2-I)A^{-1}t_2^2$ as the vector \[\left(\begin{array}{cc}0&0\\0&-2 \end{array}\right)\left(\begin{array}{cc}d&-b\\-c&a \end{array}\right)\left(\begin{array}{c}1\\0 \end{array}\right)=\left(\begin{array}{c}0\\2c \end{array}\right),\] which represents $(x_2)^{2c}$.  Thus, to prove the proposition, it suffices to show that $w\in \langle (x_2)^{2c}\rangle$.

For any $v\in \Z^2$, we have $v(t_1t_2)v^{-1}=t_1t_2((T_2T_1-I)v)$. Since $\phi(t_1t_2)=t_1t_2w$ and $\phi$ is conjugacy-preserving, this means that $w\in \im(T_2T_1-I)$. With respect to the basis $\{t_2^2, x_2\}$, $T_2T_1-I$ is represented by the matrix:\[N:=BABA^{-1}-I= \left(\begin{array}{cc}ad+bc-1&-2ab\\-2cd&ad+bc-1 \end{array}\right)=\left(\begin{array}{cc}2bc&-2ab\\-2cd&2bc \end{array}\right), \] where we have used that $ad-bc=1$. Then $\det(N)=4(b^2c^2-abcd)=4bc(bc-ad)=-4bc$ which is nonzero since $T_2T_1=(T_1T_2)^{-1}$ is Anosov. To find the vector $y_2$ whose image is $x_2$ we compute \[y_2=N^{-1}\left(\begin{array}{c}0\\1 \end{array}\right)=\frac{1}{-4bc}\left(\begin{array}{cc}2bc&2ab\\2cd&2bc \end{array}\right)\left(\begin{array}{c}0\\1 \end{array}\right)=\frac{1}{-4bc}\left(\begin{array}{c}2ab\\2bc \end{array}\right)=\frac{1}{-2c}\left(\begin{array}{c}a\\c \end{array}\right).\] Since $ad-bc=1$, $a$ and $c$ are coprime. 
Thus any integral vector that is a multiple of $y_2$ is also a multiple of $2cy_2$.
But this implies $w\in \langle (x_2)^{2c}\rangle$, as desired.      \end{proof}

\section{Proof of Theorem~\ref{thm:irred-rigid}}\label{sec:proof-irred}

Let $M$ be a closed, smooth, orientable manifold. 
We will consider the following topological groups of self-maps of $M$, and briefly review the various relationships between them:
\begin{itemize}
    \item $\Diff(M)$: the group of self-diffeomorphisms of $M$ endowed with the $C^\infty$ topology.
    \item $\Homeo(M)$: the group of self-homeomorphisms of $M$, endowed with the compact open topology.
    \item $\HE(M)$: the group of self-homotopy equivalences of $M$, endowed with the compact open topology.
\end{itemize}
Recall that for any topological group $G$, the path component of the identity, denoted by $G_0$, is a normal subgroup and thus the set of path-components $\pi_0( G)=G/G_0$ inherits a group structure. The natural inclusions \[\Diff(M)\hookrightarrow \Homeo(M)\hookrightarrow \HE(M)\] induce homomorphisms on the groups of path components \[\pi_0(\Diff(M))\to\pi_0(\Homeo(M))\to\pi_0(\HE(M)).\]

Fixing a basepoint $p\in M$, a homotopy equivalence $f\colon M\rightarrow M$ 
induces an isomorphism $f_*\colon \pi_1(M,p)\rightarrow\pi_1(M,f(p))$. 
 Choosing a path $\gamma$ from $p$ to $f(p)$, we obtain an automorphism $\phi_\gamma \circ f_* \in \Aut(\pi_1(M,p))$ where for $[\alpha]\in \pi_1(M,f(p))$, we let $\phi_\gamma([\alpha])=[\bar{\gamma}\cdot\alpha\cdot\gamma]$. Different choices of $\gamma$ lead to automorphisms that differ by an inner automorphism of $\pi_1(M,p)$; hence we obtain a well-defined map $\HE(M)\to \Out(\pi_1(M))$ which does not depend on the basepoint $p$. Furthermore, since homotopic maps induce the same automorphism up to conjugacy, this homomorphism descends to a homomorphism $\pi_0(\HE(M))\to\Out(\pi_1(M))$. Thus we may extend our sequence one step further to obtain:
\[ \pi_0(\Diff(M))\xrightarrow[]{i}\pi_0(\Homeo(M))\xrightarrow[]{j}\pi_0(\HE(M))\xrightarrow[]{k} \Out(\pi_1(M))\]

When $M$ is a 3-manifold, $i$ is an isomorphism. Indeed, Moise \cite{Moise-affineV,MoiseVIII} and Bing \cite{Bing-LocallyTame} showed that every 3-manifold admits a PL triangulation and that every homeomorphism of $M$ is approximable by a PL homeomorphism. Cairns \cite{Cairns-triangulation} showed that any PL 3-manifold is smooth, and that any PL homeomorphism of $M$ can be approximated by a diffeomorphism (see also the work of Munkres \cite{Munkres1959,Munkres1960}). These results imply that  $\Diff(M)$ is dense in $\Homeo(M)$. Building on these results Cerf \cite{Cerf-embeddings} showed that if the Smale conjecture holds, then $\Diff(M)\hookrightarrow \Homeo(M)$ is a weak homotopy equivalence.  The Smale conjecture was later confirmed by Hatcher \cite{Hatcher-Smale}, so in fact $\pi_k(\Diff(M))\to\pi_k(\Homeo(M))$ is an isomorphism for all $k$ and in particular, for $k=0$.

As in the introduction, we denote by $\Mod^+(M)$ the \emph{mapping class group} of $M$, that is, the set of orientation-preserving diffeomorphisms of the smooth, closed, oriented $3$-manifold $M$, modulo isotopy.
We define the \emph{full mapping class group of $M$} to be $\Mod(M):=\pi_0(\Diff(M))$ or equivalently, $ \pi_0(\Homeo(M))$ according to the preceding discussion. Similarly, we define the \emph{homotopy mapping class group of $M$} to be $\HMod(M):=\pi_0(\HE(M))$.

The following is a summary of the work of many authors, referring to the maps in the above diagram.  For more details, see \cite[Section~3]{HM:2013}.

\begin{theorem}\label{thm:mappingclass}Let $M$ be an irreducible, closed, oriented $3$-manifold.  Then
\begin{enumerate}[font=\upshape]
\item\label{item: mapping class main} If $M\neq S^3,\R P^3$, then $\Mod(M)$ injects into  $\Out(\pi_1(M))$.
\item\label{item: mapping class S3 or RP3} If $M=S^3$ or $\R P^3$, then $\Mod(M)\cong \HMod(M)\cong \Z/2$.
\end{enumerate}
In either case, $\Mod^+(M)\hookrightarrow \Out(\pi_1(M))$.
\end{theorem}
\begin{proof} 
If $\Gamma = \pi_1(M)$ is finite, then by geometrization $M$ is finitely covered by $S^3$ \cite{Perelman:2002-1, Perelman:2003-2, Perelman:2003-1}.  If $M\neq S^3,\R P^3$, then $\Mod(M)$ injects into $\Out(\Gamma)$ \cite{McCullough:2002-1}. When $M = S^3$ or $\R P^3$, the group $\Out(\Gamma)$ is trivial. Theorems of Fisher~\cite{Fisher} and Cerf~\cite{Cerf:1968} in the case of $S^3$, and Bonahon~\cite{Bonahon:1983} in the case of $\R P^3$, imply that $\Mod(M) \cong \HMod(M) \cong \Z/2$, generated by an orientation-reversing diffeomorphism.
  This gives result \pref{item: mapping class S3 or RP3}. Additionally, $\Mod^+(M)$ is trivial and thus injects into $\Out(\Gamma)$. 

If $M$ is irreducible and $\Gamma$ is infinite, then $M$ is aspherical, hence $\HMod(M)\cong \Out(\Gamma)$, by, for example,~\cite[Theorem~1B.9]{Hatcher:2002}.    There are two cases to consider, when $M$ is Haken and when $M$ is non-Haken.  If $M$ is Haken, a well-known theorem of Waldhausen \cite{W3} implies that any homotopy equivalence is homotopic to a homeomorphism, and that any two homotopic homeomorphisms are isotopic.  This implies that $\Mod(M)\cong\HMod(M)$, proving \pref{item: mapping class main} in this case. 

When $M$ is non-Haken, it is either hyperbolic or Seifert-fibered over $S^2$ with 3 exceptional fibers. Suppose that $M$ is hyperbolic and non-Haken. Then Gabai--Meyerhoff--Thurston \cite{GMT2003} showed that $\Mod(M)\cong\HMod(M)$ and by Mostow rigidity~\cite{mostow-rigidity}, we know $\HMod(M)\cong\Out(\pi_1(M))$; hence \pref{item: mapping class main} holds in this case. In the non-Haken Seifert-fibered case, work of Scott \cite{Scott85} and Boileau--Otal \cite{BO1991} proved that $\Mod(M)\cong \HMod(M)$ and that $\Mod(M)$ injects into $\Out(\Gamma)$, proving \pref{item: mapping class main}.

For the final statement when $M\neq S^3, \R P^3$, note that either $\Mod^+(M)$ is an index two subgroup of $\Mod(M)$ if $M$ admits an orientation-reversing self-diffeomorphism, or equal to $\Mod(M)$ otherwise. In either case, we see that $\Mod^+(M)\hookrightarrow \Out(\Gamma)$.  
\end{proof}

\noindent We now use Theorem~\ref{thm:PropertyA} and Theorem \ref{thm:mappingclass} to prove Theorem \ref{thm:freehomotopy}. 

\begin{reptheorem}{thm:irred-rigid}
Let $M$ be an irreducible, smooth, closed, oriented $3$-manifold. If an orientation-preserving diffeomorphism $f\colon M\to M$ preserves  homotopy classes of loops, that is, for any loop $\gamma\subseteq M$ the loop $f(\gamma)$ is  homotopic to $\gamma$, then $f$ is isotopic to the identity. 
\end{reptheorem}

\begin{proof}Let $M$ be an irreducible, smooth, closed, oriented 3-manifold, and let $\Gamma=\pi_1(M)$. Suppose $f$ is  an orientation-preserving diffeomorphism of $M$ which preserves  homotopy classes of loops, and denote by $\overl{f}\in \Mod(M)$ the isotopy class of $f$. Since $f$ preserves  homotopy classes of loops, it induces a class preserving element of $\Out(\pi_1(M))$. By Theorem~\ref{thm:PropertyA} this induced map is an inner automorphism and thus trivial in $\Out(\pi_1(M))$. By Theorem~\ref{thm:mappingclass}, we know that $\Mod^+(M)\hookrightarrow \Out(\pi_1(M))$ and thus $\overline{f}$ is trivial in $\Mod^+(M)$, as required. 
\end{proof}

\section{\texorpdfstring{The Gluck twist on knots in $S^1\times S^2$ with odd winding number}{The Gluck twist on knots in S1xS2 with odd winding number}}\label{sec:glucktwist}\label{sec:odd}
We now turn our attention to the effect of the Gluck twist on knots in $S^1\times S^2$.  Recall that the winding number of a knot $K\subseteq S^1 \times S^2$ is the algebraic intersection number of $K$ and $\{pt\}\times S^2$.  

In order to consistently draw pictures of knots in $S^1\times S^2$ we introduce some notation.  Let $U$ be the standard unknot in $S^3$ and $D$ be a knot in $S^3\smallsetminus U$.  Performing $0$-framed Dehn surgery on $S^3$ along $U$ results in $S^1\times S^2$.  If the image of $D$ is isotopic to $K$ in $S^1\times S^2$ then we will say that $D$ is a \emph{diagram} for $K$.  We will denote by $m_D$ the meridian of $D$ and by $h_D$, the meridian of $U$ (see Figure~\ref{fig:BasisCurves}). Recall from the introduction that adding a full right-handed twist to $D$ gives a diagram of $\G(K)$.

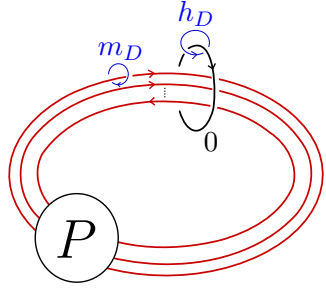
\begin{figure}[htb]
\begin{tikzpicture}
\node at (0,0) {\includegraphics[width=.3\textwidth]{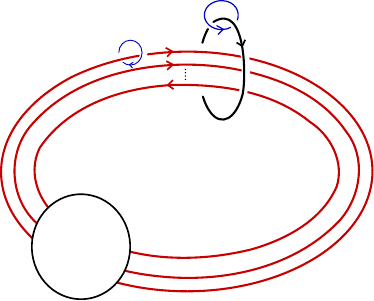}};
\node at (-0.1,1.9) {\textcolor{blue}{$h_{D}$}};
\node at (-1.3,1.4) {\textcolor{blue}{$m_{D}$}};
\node at (.6,0.2) {$0$};
\node at (-1.4,-1.3) {\huge $P$};
\end{tikzpicture}
\caption{A diagram $D$ of a knot $K\subseteq S^1\times S^2$ together with curves $m_{D}$ and $h_{D}$ forming a basis for the first homology of $S^1\times S^2\smallsetminus K$.}\label{fig:BasisCurves}\end{figure}

Note that $D$ and ${D'}$ are diagrams for isotopic knots in $S^1\times S^2$ if and only if $D$ and $D'$ are related in $S^3$ by a sequence of isotopies in the exterior of $U$ and  slides over the $0$-framed knot $U$.  We refer to the latter move as a \emph{handleslide} for convenience, even though, technically speaking, it is an isotopy of the knot $K$ in $S^1\times S^2$. See Figure~\ref{fig:slides} for a depiction of positive and negative handleslides.

\begin{figure}[htb]
\begin{tikzpicture}
\node at (0,0) {\includegraphics[width=.75\textwidth]{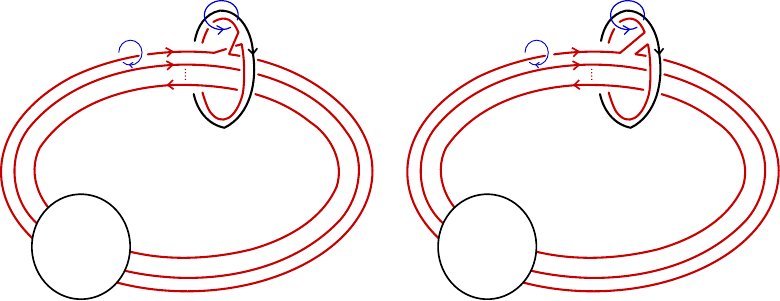}};
\node at (-4.7,1.7) {\textcolor{blue}{$m_{D}$}};
\node at (1.7,1.7) {\textcolor{blue}{$m_{D}$}};
\node at (3.1,2.25) {\textcolor{blue}{$h_{D}$}};
\node at (-3.4,2.25) {\textcolor{blue}{$h_{D}$}};
\node at (4.2,.2) {$0$};
\node at (-2.2,.2) {$0$};
\node at (-4.9,-1.5) {\huge $P$};
\node at (1.55,-1.5) {\huge $P$};
\end{tikzpicture}
\caption{Positive (left) and negative (right) handleslides performed from $D$ together with the image of $m_{D}$ and $h_{D}$ after the handleslides. To make sense of positive versus negative, observe that after a positive (respectively, negative) handleslide, the new knot $D'$ has linking number $+1$ (respectively, $-1$) with $h_D$. Moreover, we remark that a handleslide is not a well-defined move on a diagram, even after fixing the sign, since there are various choices involved in performing band surgery on $D$ with a $0$-framed push-off of $U$.}\label{fig:slides}
\end{figure}


A framing of a knot $K$ is an identification of a regular neighbourhood $\nu(K)$ with $S^1\times D^2$. Equivalently, a framing is simply a choice of push-off $K^+$ of $K$.  Given a choice of diagram $D$ for $K$ and an integer $a$, the \emph{$a$-framing of $K$ with respect to $D$} is specified by taking the $a$-framing on $D$.  In other words the $a$-framing is given by the push-off $D^+$ so that the linking number $\lk(D,D^+)=a$, where the latter is computed in $S^3$. 
Note that this really does depend on the choice of $D$, as a handleslide changes $a$
(see Figure~\ref{fig:framing}), but we avoid this by fixing a diagram of a knot. Let $M(D, a)$ denote the $3$-manifold obtained by modifying $S^1\times S^2$ by $a$-framed surgery along $D$ where $D$ is a diagram for $K$. Explicitly, this is the $3$-manifold obtained by performing $(0,a)$ framed surgery on $S^3$ along the $2$-component link $(U,D)$. For simplicity, we still use the letters $m_D$ and $h_D$ to denote the images of these curves in the surgered manifold.

\begin{figure}[htb]
\begin{tikzpicture}
\node at (0,0) {\includegraphics[width=.75\textwidth]{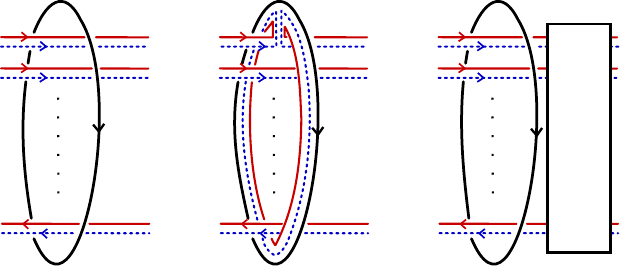}};
\node at (-4.3,2.3) {$0$};
\node at (0.1,2.3) {$0$};
\node at (4.5,2.3) {$0$};
\node at (5.4,0) {\Large{+1}};
\end{tikzpicture}
\caption{Left: The push-off (blue, dashed) corresponding to the $a$-framing on a diagram $D$ (red, solid) for some knot $K$ with winding number $w$. Center: The $a$-framed push-off is sent to to the $(a-2w)$-framed push-off of a diagram $D'$ for $K$ by a negative handleslide. 
Right: The $a$-framed push-off is sent to the $(a+w^2)$-framed push-off of a diagram $\widetilde{D}$ for $\G(K)$ by the Gluck twist.}\label{fig:framing}
\end{figure}
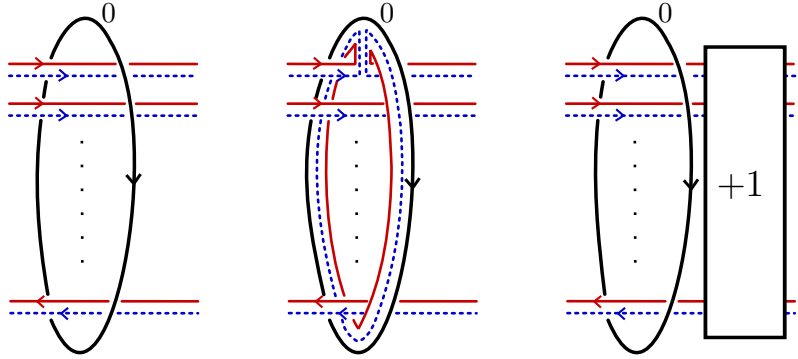

\begin{proposition}\label{prop:prelim}
Let $K$ be a knot in $S^1\times S^2$ with winding number $w$ and $D$ be a diagram of $K$. Suppose $\widetilde{D}$ is the diagram of $\G(K)$ obtained from $D$ by adding a full right-handed twist. Suppose that $D$ can be obtained from $\widetilde{D}$ by a sequence of isotopies in the exterior of $U$, $k_+$ positive handleslides, and $k_-$ negative handleslides. Let $k=k_+-k_-$.


Then there is a self-diffeomorphism of $S^{1}\times S^{2}$ that fixes $D$ and, for every integer $a$, sends the $a$-framing of $D$ to its $(a+w^{2}+2kw)$-framing.

Consequently, for any integer $a$, there is a diffeomorphism 
\[
\phi\colon M(D,a)\to M\left(D,{a+w^2+2kw}\right).
\]
Moreover, we have that $$\phi_*:H_1(M(D,a);\Z)\to H_1\left(M\left(D,{a+w^2+2kw}\right);\Z\right)$$ is given by 
\[
\phi_*([m_{D}]) = [m_{D}] \text{ and }\phi_*([h_{D}]) = [h_{D}]+(w+k)\cdot [m_{D}].
\]
\end{proposition}

\begin{proof} As in the hypothesis, let $D$ and $\widetilde{D}$ be diagrams of $K$ and $\G(K)$, respectively, where $\widetilde{D}$ is obtained from $D$ by adding a full right-handed twist. The far right panel of Figure~\ref{fig:framing} reveals that the Gluck twist sends the $a$-framing of $K$ with respect to $D$ to the $(a+w^2)$-framing of $\G(K)$ with respect to $\widetilde{D}$. Moreover, Figure~\ref{fig:Twist-With-Basis} shows that $\G$ sends $m_{D}$ to $m_{\widetilde{D}}$ and sends the homology class $[h_{D}]$ to $[h_{\widetilde{D}}]+w\cdot [m_{\widetilde{D}}]$.

\begin{figure}[htb]
\begin{tikzpicture}
\node at (0,0) {\includegraphics[width=.75\textwidth]{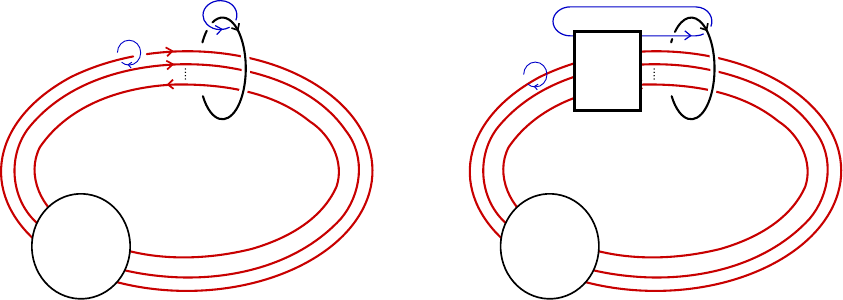}};
\node at (-4.9,1.6) {\textcolor{blue}{$m_{D}$}};
\node at (0.8,1.4) {\textcolor{blue}{$\G(m_{D})$}};
\node at (-3.5,2.1) {\textcolor{blue}{$h_{D}$}};
\node at (1.2,2.1) {\textcolor{blue}{$\G(h_{D})$}};
\node at (3.6,.2) {$0$};
\node at (-2.3,.2) {$0$};
\node at (-5,-1.4) {\huge $P$};
\node at (1.9,-1.4) {\huge $P$};
\node at (2.75,1.2) {\Large $+1$};
\end{tikzpicture}
\caption{The Gluck twist together with the image of basis curves $m_{D}$ and $h_{D}$.}\label{fig:Twist-With-Basis}
\end{figure}
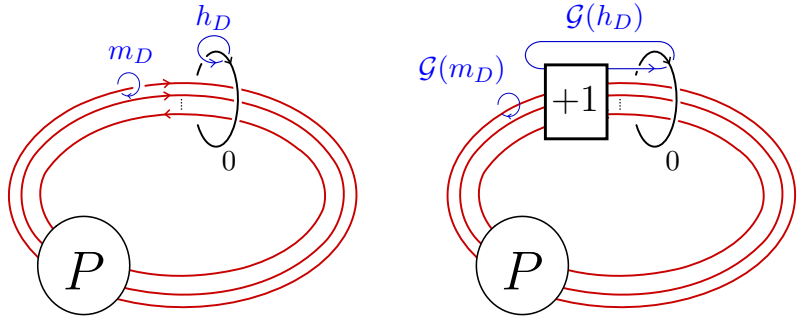

Let $D$ be a diagram for a knot $K$ in $S^1\times S^2$.  Let $D'$ be the result of modifying $D$ by a negative
 handleslide.  Figures~\ref{fig:slides} and \ref{fig:framing} show that this handleslide sends the $a$-framing of $K$ with respect to $D$ to the 
 $(a-2w)$-framing
 of $K$ with respect to the new diagram $D'$, sends the meridian $m_D$ to the meridian $m_{D'}$, and sends the homology class $[h_D]$ to 
 $[h_{D'}]-[m_{D'}]$.
If instead $D'$ is the result of a positive handleslide, then the $a$-framing is sent to the $(a+2w)$-framing, $[m_D]$ is sent to $[m_{D'}]$, and $[h_D]$ is sent to $[h_{D'}]+[m_{D'}]$.
 Naturally, isotopies of a diagram in the exterior of $U$ do not change the framing.

By assumption, the diagram $\widetilde{D}$ can be changed to $D$ by a sequence of isotopies in the exterior of $U$, $k_+$ positive handleslides, and $k_-$ negative handleslides. Recall that $k=k_+-k_-$. Composing $\G$ with the diffeomorphism of $S^1\times S^2$ induced by the claimed isotopy, we get the desired self-diffeomorphism of $S^1\times S^2$ sending $D$ to itself, and specifically the $a$-framing of $D$ to the $(a+w^2+2kw)$-framing of $D$. By construction this map sends $[m_D]$ to $[m_D]$, and sends $[h_D]$ to $[h_D]+(w+k)\cdot[m_D]$. The desired map $\phi$ is obtained by performing surgeries dictated by the framings.\end{proof}

We now begin the proof of Theorem~\ref{thm:glucktwist}. In this section we address the case of odd winding numbers. The following is the goal of the remainder of this section. 


\begin{proposition}\label{prop:w=-2k}
Let $K$ be a knot in $S^1\times S^2$ with winding number $w$ and $D$ be a diagram of $K$. Suppose $\widetilde{D}$ is the diagram of $\G(K)$ obtained from $D$ by adding a full right-handed twist. Suppose that $D$ can be obtained from $\widetilde{D}$ by a sequence of isotopies in the exterior of $U$, $k_+$ positive handleslides, and $k_-$ negative handleslides. Let $k=k_+-k_-$.

Then either $w^2+2kw=0$ or $K$ is the Hopf knot or its reverse.
\end{proposition}

The above will settle Theorem~\ref{thm:glucktwist} for odd winding numbers since $w^2+2kw=0$ implies that $w$ is even. 

We will need the following definitions. A diffeomorphism $f$ of a compact $3$-manifold $M$ is said to be \emph{Dehn twists on the boundary} if its restriction to the boundary $\partial M$ is isotopic to the identity on the complement of a collection of disjoint simple closed curves in $\partial M$. If this collection is nonempty, and
the restricted diffeomorphism is not isotopic to the identity on the complement of any
proper subset of the collection, then we say that $f$ is Dehn twists \emph{about} the collection. The restriction of $f$ to $\partial M$ is then isotopic to a composition of nontrivial (powers of) Dehn twists about the curves. We will use the following result of McCullough.

\begin{theorem}[{\cite[Theorem~1]{McCullough06}}]\label{thm:mccullough}
Let $M$ be a smooth, compact, orientable $3$–manifold which admits a diffeomorphism which is Dehn twists on the boundary about the collection $C_1,\dots, C_n$ of simple closed curves in $\partial M$. Then for each $i$, either $C_i$ bounds a disk in $M$, or for some $j\neq i$, $C_i$ and $C_j$ cobound an incompressible annulus in M.
\end{theorem}

%

\begin{proof}[Proof of Proposition~\ref{prop:w=-2k}]
Proposition~\ref{prop:prelim} implies that there is a self-diffeomorphism of $S^{1}\times S^{2}$ that fixes $D$ and, for every integer $a$, sends the $a$-framing of $D$ to its $(a+w^{2}+2kw)$-framing. By
construction, this diffeomorphism preserves a regular neighbourhood of $D$. On the boundary of this solid torus neighbourhood, the induced map is the identity if $w^2+2kw=0$ and the composition of $w^2+2kw$ Dehn twists along $m_D$ otherwise.  

By Theorem~\ref{thm:mccullough}, if $w^2+2kw\neq 0$, then $m_D$ bounds an embedded disk $\Delta$ in $S^1\times S^2\smallsetminus K$. The union of $\Delta$ with a meridional disk for $\nu(K)$ is a nonseparating sphere in $S^1\times S^2$ which intersects $K$ precisely once. Since the sphere intersects $K$ precisely once, it represents a generator of $H_2(S^1 \times S^2;\Z)\cong \Z \cong \pi_2(S^1 \times S^2)$. We choose its orientation so that it agrees with the class represented by the inclusion $\{\mathrm{pt}\}\times S^2 \hookrightarrow S^1 \times S^2$. It follows that the sphere is homotopic to $\{\mathrm{pt}\}\times S^2$, and hence isotopic to it as well~\cite{laudenbach:annals-1973}. By the isotopy extension theorem and the lightbulb trick, $K$ is isotopic to the Hopf knot or its reverse.
\end{proof}


The above arguments can also be used to prove the analogue of Theorem~\ref{thm:main} for connected sums of $S^1\times S^2$, as follows. 

\begin{corollary}\label{cor:nS1xS2}
    Let $M=\#^n (S^1\times S^2)$ for some $n\geq 1$ and suppose that $f\colon M\to M$ is an orientation-preserving diffeomorphism which is not isotopic to the identity. Then there is a knot $K\subseteq M$ such that $K$ and $f(K)$ are not isotopic. 
\end{corollary}

\begin{proof}
If $f$ does not preserve homotopy classes of loops, then we may choose $K$ to be any knot whose homotopy class is not preserved by $f$. Here we used the fact that since the ambient manifold is 3-dimensional, every element of the fundamental group can be represented by an embedded loop. For the rest of the proof, assume that $f$ preserves homotopy classes of loops, i.e.~the induced map on $\pi_1(M)$ is class preserving. By Theorem~\ref{thm:PropertyA}, the fundamental group $\pi_1(M)$ has Property~A, so $f$ lies in the kernel of $\Phi\colon\Mod^+(M)\to \Out(\pi_1(M))$. Laudenbach showed in~\cite{laudenbach:annals-1973,laudenbach:homotopie-et-isotopie} that the kernel of $\Phi$ is isomorphic to $(\Z/2)^n$, generated by Gluck twists along the nonseparating $2$-spheres $*\times S^2$ in each summand. So, up to isotopy, $f$ is the composition of the Gluck twists along some nontrivial subset of the standard nonseparating $2$-spheres $*\times S^2$ in each summand. Up to reordering, we assume that $S$ lies in this subset, where $S$ denotes the nonseparating sphere in the first $S^1\times S^2$ summand in $M$.


Choose a knot $K\subseteq M$ in the homology class $(w_1, 0, \dots, 0)\in \Z^n\cong H_1(M;\Z)$ with $w_1\geq 3$ and odd. Then arguing as in Propositions~\ref{prop:prelim} and \ref{prop:w=-2k}, we will see that $K$ is not isotopic to $f(K)$. Consider $M$ as obtained from $0$-framed surgery on the $n$-component unlink. Then $K$ is represented by a diagram $D$ with linking number $w_1$ with the first component, and zero with every other component. Fixing a framing $a$ on $D$, by the proof of Proposition~\ref{prop:prelim}, we see that if $K$ were isotopic to $f(K)$, then there would be a self-diffeomorphism of $M$ taking the $a$-framing of $D$ to the $(a+w_1^2+2kw_1)$-framing, for some integer $k$. Note that $f$ may involve Gluck twists along other nonseparating spheres in $M$ and the putative isotopy might involve handleslides interacting with other parts of the surgery diagram. But since $K$ has trivial winding number with all but the first component of the surgery diagram, there is no contribution of such moves on the framing. Returning to the proof, we see, as in Proposition~\ref{prop:w=-2k}, using Theorem~\ref{thm:mccullough}, that either $w_1^2+2kw_1=0$ or $K$ intersects $S$ precisely once, both of which conditions contradict our hypotheses on $K$. Therefore $K$ is not isotopic to $f(K)$.
\end{proof}

Corollary~\ref{cor:nS1xS2} for $n=1$ and Theorem~\ref{thm:freehomotopy} together imply Theorem~\ref{thm:main}. At present Corollary~\ref{cor:nS1xS2} is our only application to reducible $3$-manifolds. The following remark explains the difficulty in extending our proof techniques to general reducible $3$-manifolds.

\begin{example}
While our paper is focused primarily on knots, our techniques may be applied to the study of links. We illustrate this principle now. Consider the diagram $D$ of the $n$-component link $L$ in $S^1\times S^2$, consisting of the closure of the trivial $n$-strand braid (see Figure~\ref{fig:trivial-braid-link}).
\begin{figure}[htb]
\begin{tikzpicture}
\node at (0,0) {\includegraphics[width=.3\textwidth]{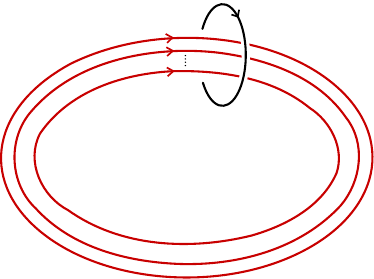}};
\node at (1.1,1.6) {$0$};
\end{tikzpicture}
\caption{A diagram $D$ of a link $L$ in $S^1\times S^2$ obtained as the closure of the trivial $n$-strand braid.}\label{fig:trivial-braid-link}
\end{figure}
Let $M(D;f_1,f_2,\dots, f_n)$ denote the manifold obtained by performing $(f_1,f_2,\dots, f_n)$-framed surgery on $S^1\times S^2$ along $D$, or more explicitly, by performing $(0;f_1,f_2,\dots, f_n)$-framed surgery on $S^3$ along the link $(U;L)$. By the proof of Proposition~\ref{prop:prelim} if the isotopy class of $L$ is preserved under the Gluck twist, there is a diffeomorphism $M(D; 0,0,\dots,0)\cong M(D;1+2x_1,1+2x_2,\dots,1+2x_n)$ for integers $x_1,x_2,\dots,x_n\in \Z$. By a sequence of handleslides and removing a 0-framed Hopf link, it is clear that $M(D; 0,0,\dots,0)\cong \#^{n-1} S^1\times S^2$. On the other hand, the manifold $M(D;1+2x_1,1+2x_2,\dots,1+2x_n)$ is a Seifert-fibered space (see e.g.\ \cite{Neumann:1978-1}). The only non-prime orientable Seifert-fibered space is $\mathbb{R}P^3\# \mathbb{R}P^3$~\cite[p.\ 9]{Friedlbook}. Thus for $n\geq 3$, the isotopy class of $L$ is nontrivially altered by the Gluck twist. This completes the isotopy classification of closed surface braids by Grant-Sienicka~\cite{Grant_2020} (see~\cite[Remark~1.4]{Grant_2020}).
\end{example}
\section{Non-zero even winding numbers}\label{sec:even}

In this section, we show that for each positive even integer $w$, there exists a knot $K$ with winding number $w$ whose isotopy class is not preserved under the Gluck twist. This will imply Theorem~\ref{thm:glucktwist} for all $w\neq 0$ since the result of changing the orientation of any knot with winding number $w$ has winding number $-w$.

First, we consider the case when $w$ is greater than $2$ and at the end of the section we deal with the winding number $2$ case. For relatively prime integers $p,q$, we denote the $(p,q)$ torus knot by $T_{p,q}$. As in Section~\ref{sec:glucktwist}, given a knot $K$ in $S^1\times S^2$ and a diagram $D$ for $K$, the meridian of $K$ is denoted by $m_{D}$ and the meridian of the unknotted surgery curve $U$ is denoted by $h_{D}$. 

\begin{figure}[htb]
\includegraphics[width=.4\textwidth]{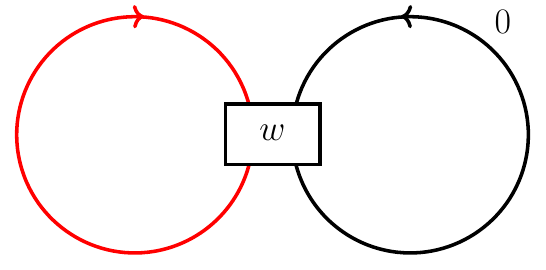}
\caption{A diagram $D_w$ for the knot $K_w$.  The box containing $w$ indicates $w$ full right-handed twists. 
}\label{fig:evenwinding}
\end{figure}

\begin{lemma}\label{lem:surgtorus}Let $K_w\subseteq S^1\times S^2$ be the knot with diagram $D_w$ and winding number $w$ shown in Figure~\ref{fig:evenwinding}. If $w$ is an even integer and $K_w$ is isotopic to $\G(K_w)$, then there exists a diffeomorphism
$$\psi \colon S^3_{w^2}(T_{w,w+1}) \to S^3_{w^2}(T_{w,w+1}),$$
such that, if $\mu$ is the meridian of $T_{w,w+1}$, then $$\psi_*\colon H_1\left(S^3_{w^2}(T_{w,w+1});\Z\right)\to H_1\left(S^3_{w^2}(T_{w,w+1});\Z\right)$$ is given by $$\psi_*([\mu]) = \left(\frac{w^2}{2}+1\right)\cdot[\mu].$$
\end{lemma}

\begin{proof} Suppose $K_w$ is isotopic to $\G(K_w)$ and let $D_w$ be the diagram of $K_w$ described in Figure~\ref{fig:evenwinding}. Then by Propositions~\ref{prop:prelim} and~\ref{prop:w=-2k}, there is a diffeomorphism 
\[
\phi\colon M(D_w,a)\to M\left(D_w,a\right)
\]
for each $a\in \Z$, where $$\phi_*:H_1(M(D_w,a);\Z)\to H_1\left(M\left(D_w,{a}\right);\Z\right)$$ satisfies $$\phi_*([m_{D_w}]) = [m_{D_w}] \qquad \text{ and } \qquad \phi_*([h_{D_w}]) = [h_{D_w}]+\frac{w}{2}\cdot [m_{D_w}].$$
Set $a=-1$. Now observe that we have a surgery description of $M(D_w,-1)$ as the $(-1,0)$ framed surgey on $S^3$ along the link $(D_w,U)$ (see Figure~\ref{fig:evenwinding}). In particular, we may blow down $D_w$ since it is an unknotted circle with framing $-1$. The second component is transformed into $T_{w,w+1}\subseteq S^3$ with framing $w^2$ (see, for example,~\cite[p.\ 150-151]{gompf-stipsicz}). 

To see the second statement, it suffices to note that we have the relation $[m_{D_w}]=w[h_{D_w}]$ in $H_1(M(D_w,-1);\Z)$ and the curve $h_{D_w}$ is mapped to $\mu$ under the blowdown.
\end{proof}

In order to show that the diffeomorphisms claimed in the above lemma do not exist, we use the Heegaard-Floer correction term $d(M,\mathfrak{t})\in \mathbb{Q}$ associated to a rational homology sphere $M$ with a Spin$^c$ structure~$\mathfrak{t}$~\cite{Ozsvath-Szabo:2003-2}. Recall that for any knot $K$ in $S^3$, there is a non-increasing sequence of non-negative integers $\{V_i(K)\}_{i \geq 0}$, introduced by Rasmussen \cite{Rasmussen:2003-1}. Ni and Wu showed that the correction term of $3$-manifolds obtained by surgeries on knots can be computed using this sequence \cite{Ni-Wu:2015-1}. Here we only state the formula for the integral surgery.

\begin{proposition}[{\cite[Proposition 1.6 and Remark 2.10]{Ni-Wu:2015-1}}]\label{proposition:Viformula} If $n$ is a positive integer and $U$ is the unknot, then for any knot $K$,
$$ d(S^3_n(K),\mathfrak{t}_i) = d(S^3_n(U),\mathfrak{t}_i)-2\max \{V_{i}(K),V_{n-i}(K)\} .$$
\end{proposition}
Here we are using the identification $\varphi\colon \operatorname{Spin}^c(S^3_{n}(K))\to \{0,1, \dots, n-1\}$ given in \cite{Ozsvath-Szabo:2011-1} so that the $\operatorname{Spin}^c$ structure that corresponds to $i\in  \{0,1, \dots, n-1\}$ under $\varphi$ is denoted by $\mathfrak{t}_i$. We recall some facts about this identification (see \cite[Appendix~B]{Cochran-Horn:2015-1}). A free transitive action by $H^2
(S^3_{n}(K); \Z)$ on $\operatorname{Spin}^c(S^3_{n}(K))$, denoted by $\mathfrak{t} + x$ for $x \in H^2
(S^3_{n}(K); \Z)$ and $\mathfrak{t} \in  \operatorname{Spin}^c(S^3_{n}(K))$, is given as follows
$$\mathfrak{t}_j = \mathfrak{t}_i + (j-i)\cdot PD[[\mu]],$$
where $\mu$ is the positively oriented meridian of $K$. When $n$ is even, the spin structures of $S^3_{n}(K)$ are $\mathfrak{t}_0$ and $\mathfrak{t}_{\frac{n}{2}}$ among the $\operatorname{Spin}^c$ structures of $S^3_{n}(K)$. Moreover, the correction terms for lens spaces are computed in \cite{Ozsvath-Szabo:2003-2} for all Spin$^c$ structures. In particular, we have the following.

\begin{proposition}[{\cite[Proposition 4.8]{Ozsvath-Szabo:2003-2}}]\label{prop:dlens} If $n$ is a positive integer and $U$ is the unknot, then
$$ d(S^3_n(U),\mathfrak{t}_i) = \frac{(n-2i)^2}{4n}-\frac{1}{4}.$$
\end{proposition}

Recall that for the $(p,q)$-torus knot $T_{p,q}$, we can define the gap
counting function $I_{p,q}$ of the semigroup $\Gamma_{p,q} = \langle p,q \rangle =\{ip+jq\mid i,j\geq 0\}
\subseteq \mathbb{Z}_{\geq 0}$ as follows
$$I_{p,q}(j) = \#\left(\mathbb{Z}_{\geq j} \smallsetminus \Gamma_{p,q}\right).$$

Borodzik and Livingston \cite{Borodzik-Livingston:2014-1} proved that $V_i$ of torus knots can be computed using the gap counting function.

\begin{proposition}[{\cite[Propositions 4.4 and 4.6]{Borodzik-Livingston:2014-1}}]\label{prop:vigap} If $p$ and $q$ are coprime positive integers, then
$$ V_j(T_{p,q}) = I_{p,q}\left(j+\frac{(p-1)(q-1)}{2}\right).$$
\end{proposition}

Using the above proposition we get the following.
\begin{proposition}\label{prop:vitorus} Let $w$ be a positive even integer. If $T_{w,w+1}$ is the $(w,w+1)$-torus knot, then
$$V_0(T_{w,w+1})= V_1(T_{w,w+1})+1 =\frac{w^2+2w}{8}\ \text{ and }\ V_{\frac{w^2}{2}-1}(T_{w,w+1})=V_{\frac{w^2}{2}}(T_{w,w+1})=0.$$
\end{proposition}

\begin{proof}It is a routine computation to see that 
$$\Gamma_{w,w+1} = \{i\cdot w+j \mid 0 \leq j \leq i \leq w-2\} \cup \mathbb{Z}_{\geq w^2-w}.$$
Combining the above equation and Proposition~\ref{prop:vigap}, we have
\begin{align*}
V_0(T_{w,w+1}) &= I_{w,w+1}\left(\frac{w^2-w}{2}\right)\\
&=\#\left(\mathbb{Z}_{\geq\frac{w^2-w}{2}} \smallsetminus \Gamma_{w,w+1}\right)\\
&=\#\bigg(\left\{i \mid \frac{w^2-w}{2} \leq i \leq w^2-w-1\right\}\\
&\phantom{=\#\bigg(} \smallsetminus \left\{i\cdot w+j \mid \frac{w}{2} \leq i \leq w-2,~ 0\leq j\leq i\right\}\bigg)\\
&=\frac{w^2-w}{2} - \frac{3w^2-6w}{8}\\ 
&=\frac{w^2+2w}{8}.
\end{align*}
\noindent Moreover, since 
$$\frac{w-2}{2}
\cdot w+\frac{w-2}{2} < \frac{w-1}{2}\cdot w <\frac{w}{2}\cdot w,$$ we see that $\frac{w-1}{2}\cdot {w} \notin \Gamma_{w,w+1}$. Hence we have
$$
V_1(T_{w,w+1}) = I_{w,w+1}\left(1+\frac{w^2-w}{2}\right)=\#\left(\mathbb{Z}_{\geq1+\frac{w^2-w}{2}} \smallsetminus \Gamma_{w,w+1}\right)=V_0(T_{w,w+1})-1.
$$
Lastly, since $w^2-w\leq \frac{2w^2-w-2}{2}$, we have  $\Z_{\geq \frac{2w^2-w-2}{2}}\subseteq\Z_{\geq w^2-w}\subseteq \Gamma_{w,w+1}$
so that
\[
V_{\frac{w^2}{2}-1}(T_{w,w+1}) =\#\left(\mathbb{Z}_{\geq\frac{2w^2-w-2}{2}} \smallsetminus \Gamma_{w,w+1}\right)=0,
\]
and since $\{V_i(K)\}_{i \geq 0}$ forms a non-increasing sequence of non-negative integers, it follows that  $V_{\frac{w^2}{2}}(T_{w,w+1})=0$.
\end{proof}

Now we are ready to show that $K_w$ and $\G(K_w)$ are not isotopic whenever $w$ is even and greater than $2$.

\begin{proposition}\label{prop:evenwinding} Let $K_w$ be the knot shown in Figure~\ref{fig:evenwinding} with even winding number $w>2$. Then $K_w$ is not isotopic to $\G(K_w)$.
\end{proposition}
\begin{proof} Suppose that $K_w$ is isotopic to $\G(K_w)$ and $w$ is even and greater than $2$. By Lemma~\ref{lem:surgtorus}, there exists a diffeomorphism 
$$\psi \colon S^3_{w^2}(T_{w,w+1}) \to S^3_{w^2}(T_{w,w+1}),$$ 
such that $\psi_*([\mu]) = \left(\frac{w^2}{2}+1\right)\cdot[\mu],$ where $\mu$ is the meridian of $T_{w,w+1}$. Note that $S^3_{w^2}(T_{w,w+1})$ has two spin structures, $\mathfrak{t}_0$ and $\mathfrak{t}_{\frac{w^2}{2}}$, and using Propositions~\ref{proposition:Viformula}, \ref{prop:dlens}, and \ref{prop:vitorus} we have
\begin{align*}
d(S^3_{w^2}(T_{w,w+1}),\mathfrak{t}_0) &= d(S^3_{w^2}(U),\mathfrak{t}_0)-2\max\{V_{0}(T_{w,w+1}), V_{w^2}(T_{w,w+1})\}\\
&= d(S^3_{w^2}(U),\mathfrak{t}_0)-2V_{0}(T_{w,w+1})\\
&=\frac{w^4}{4w^2}-\frac{1}{4}-\frac{w^2+2w}{4}\\
&=-\frac{1+2w}{4},
\end{align*}
and
\begin{align*}
d(S^3_{w^2}(T_{w,w+1}),\mathfrak{t}_\frac{w^2}{2}) &= d(S^3_{w^2}(U),\mathfrak{t}_\frac{w^2}{2})-2V_{\frac{w^2}{2}}(T_{w,w+1})\\
&=0-\frac{1}{4}-0\\
&=-\frac{1}{4}.
\end{align*}
In particular, $d(S^3_{w^2}(T_{w,w+1}),\mathfrak{t}_0) \neq  d(S^3_{w^2}(T_{w,w+1}),\mathfrak{t}_\frac{w^2}{2})$, which implies that the pull-back of $\mathfrak{t}_0$ under $\psi$ is $\mathfrak{t}_0$ \cite{Ozsvath-Szabo:2003-2}  (see also \cite[Theorem 1.2]{JaNa2016}).  Moreover, by the naturality of the action of $H^2(S^3_{w^2}(T_{w,w+1}); \Z)$ on $\operatorname{Spin}^c(S^3_n(K))$, we have
\begin{align*}
\psi^*(\mathfrak{t}_1) = \psi^* \left(\mathfrak{t}_0 + PD[[\mu]]\right)
&=\psi^* \left(\mathfrak{t}_0\right)+\psi^* \left(PD[[\mu]]\right)\\
&=\mathfrak{t}_0+\left(\frac{w^2}{2}+1\right)\cdot PD[[\mu]]\\ &=\mathfrak{t}_{\frac{w^2}{2}+1}
\end{align*}
hence $d(S^3_{w^2}(T_{w,w+1}),\mathfrak{t}_1) = d(S^3_{w^2}(T_{w,w+1}),\mathfrak{t}_{\frac{w^2}{2}+1})$. Again, by Propositions~\ref{proposition:Viformula}, \ref{prop:dlens}, and \ref{prop:vitorus} we get
\begin{align*}
d(S^3_{w^2}(T_{w,w+1}),\mathfrak{t}_1) &= d(S^3_{w^2}(U),\mathfrak{t}_1)-2\max\{V_{1}(T_{w,w+1}), V_{w^2-1}(T_{w,w+1})\}\\
&= d(S^3_{w^2}(U),\mathfrak{t}_1)-2V_{1}(T_{w,w+1})\\
&=\frac{(w^2-2)^2}{4w^2}-\frac{1}{4}-\frac{w^2+2w}{4}+2\\
&=-\frac{2w^3-3w^2-4}{4w^2},
\end{align*}
and
\begin{align*}
d(S^3_{w^2}(T_{w,w+1}),\mathfrak{t}_{\frac{w^2}{2}+1}) &= d(S^3_{w^2}(U),\mathfrak{t}_{\frac{w^2}{2}+1})-2\max\{V_{\frac{w^2}{2}-1}(T_{w,w+1}),V_{\frac{w^2}{2}+1}(T_{w,w+1})\}\\
&= d(S^3_{w^2}(U),\mathfrak{t}_{\frac{w^2}{2}+1})-2V_{\frac{w^2}{2}-1}(T_{w,w+1})\\
&=\frac{4}{4w^2}-\frac{1}{4}-0\\
&=\frac{1}{w^2}-\frac{1}{4}.
\end{align*}
A quick computation shows that $w=2$, which gives us the desired contradiction.\end{proof}

Finally, we deal with the winding number $2$ case.

\begin{figure}[htb]
\includegraphics[width=.4\textwidth]{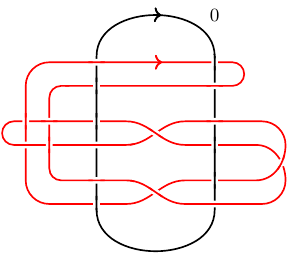}
\caption{A diagram $D$ for the winding number two knot $K$.}\label{fig:windingtwo}
\end{figure}

\begin{proposition}\label{prop:winding2} If $K$ is the winding number $2$ knot in Figure~\ref{fig:windingtwo}, then $K$ is not isotopic to $\G(K)$.
\end{proposition}
\begin{proof} Suppose $K$ is isotopic to $\G(K)$ and let $D$ be the diagram of $K$ described in Figure~\ref{fig:windingtwo}. Then by Propositions~\ref{prop:prelim} and \ref{prop:w=-2k}, there is a diffeomorphism 
$$\phi\colon M(D,1)\to M(D,1),$$ where $\phi_*:H_1(M(D,1);\Z)\to H_1(M(D,{1});\Z)$ is given by $$\phi_*([m_{D}]) = [m_{D}] \text{ and }\phi_*([h_{D}]) = [h_{D}]+ [m_{D}].$$ 
Since $[m_{D}]= 2\cdot [h_{D}]$ in $H_1(M(D,1);\Z)$, we see that $H_1(M(D,1);\Z) \cong \Z/4$ is generated by $[h_{D}]$ and $\phi_*([h_{D}]) = 3\cdot [h_{D}]$. In particular, $\phi$ is not isotopic to the identity. Using SnapPy and Sage \cite{SnapPy}, we can verify that $M(D,1)$ is hyperbolic with no non-trivial diffeomorphisms. This gives us a contradiction and completes the proof.
\end{proof}

\section{Zero winding number}\label{sec:zero}

In this section, we produce knots in $S^1\times S^2$ with zero winding number whose isotopy classes are changed by the Gluck twist. The approach was suggested to us by Charles Livingston after the preprint appeared on the arXiv. Since our examples arise as satellites, we begin by recalling the satellite construction.

Let $P$ be a knot in the solid torus $V=S^1\times D^2$ and $J$ be a knot in $S^1\times S^2$. By choosing a framing of $J$, we obtain an embedding $f_J\colon V=S^1\times D^2 \into S^1\times S^2$ which parameterizes a neighbourhood of $J$.  The knot $f_J(P)\subseteq S^1\times S^2$ is denoted by $P(J)$ and is called the \emph{satellite knot} with \emph{pattern} $P$ and \emph{companion} $J$. Note that we omit the framing in the notation, the reason for which will be clear from the next proposition.

See Figure~\ref{fig:satellite} for an example.  The following result allows us to start with a knot $J$ which is not isotopic to its Gluck twist and produce more knots with the same property. 
We will use the fact that every irreducible, orientable 3-manifold admits a \emph{JSJ-decomposition}, obtained by cutting along mutually disjoint $\pi_1$-injective, embedded tori (see, for example,~\cite[Section~1.6]{Friedlbook}), into submanifolds that are either Seifert-fibered or atoroidal. Here a connected 3-manifold $M$ is called \emph{atoroidal} if every map
$f\colon S^1\times S^1\to M$, which induces an injective homomorphism of fundamental
groups is homotopic to a map of $S^1\times S^1$ into $\partial M$.
\begin{figure}[htb]
\begin{tikzpicture}
\node at (0,0) {\includegraphics[width=.28\textwidth]{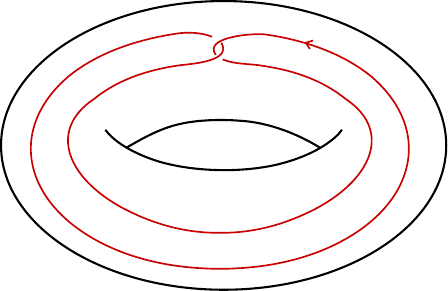}};
\node at (5,0) {\includegraphics[width=.28\textwidth]{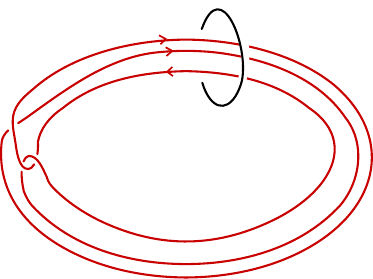}};
\node at (5.8,1.6) {$0$};
\node at (10,0) {\includegraphics[width=.28\textwidth]{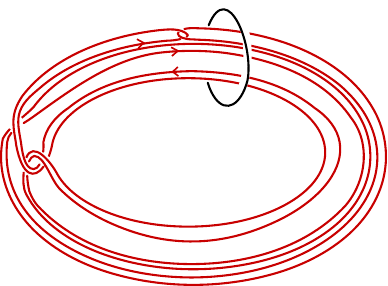}};
\node at (10.8,1.6) {$0$};
\node at (0,-2.1){$P$};
\node at (5,-2.1){$J$};
\node at (10,-2.1){$P(J)$};
\end{tikzpicture}
\caption{Left to right:  The Whitehead pattern $P$, the Mazur knot $J$ in $S^1\times S^2$, to which we assign the blackboard framing, and the satellite knot $P(J)$.}\label{fig:satellite}
\end{figure}

\begin{proposition}\label{prop:satellite}Suppose that $P$ is a knot in $V=S^1\times D^2$ and $J$ is a framed knot in $S^1\times S^2$ such that:
\begin{enumerate}
\item the exteriors $V\smallsetminus \mathring{\nu}(P)$ and $S^1\times S^2 \smallsetminus \mathring\nu(J)$ are both atoroidal; 
\item the inclusion induced maps $\pi_1(\partial V)\to \pi_1(V\smallsetminus P)$ and $\pi_1(\partial \nu(J))\to \pi_1((S^1\times S^2)\smallsetminus \mathring\nu(J))$ are injective; 
\item $J$ is not isotopic to the Hopf knot $S^1\times * \subseteq S^1\times S^2$; and 
\item $0\neq [J] \in H_1(S^1\times S^2;\mathbb{Z})$. 
\end{enumerate}
If the unframed isotopy class of $P(J)$ is not changed by the Gluck twist, then the same is true for the unframed isotopy class of $J$. Here we have added the word ``unframed'' in order to emphasize that these isotopies need not preserve any given framing.\end{proposition}
\begin{proof}
Since $J$ is framed we have a specified embedding $f_J\colon V\into S^1\times S^2$.
  Let $T=f_J(\bdry V)$.  Then $T$ splits $S^1\times S^2\smallsetminus P(J)$ into two pieces diffeomorphic to $S^1\times S^2\smallsetminus J$ and $V\smallsetminus  P$ respectively.  As $V\smallsetminus P$ and $S^1\times S^2\smallsetminus J$ are atoroidal by assumption, $S^1\times S^2\smallsetminus P(J)$ has a single JSJ torus, $T$. Here we used condition (2) to ensure that $T$ is $\pi_1$-injective in $S^1\times S^2 \smallsetminus P(J)$.

Since the Gluck twist $\G$ gives a diffeomorphism from $S^1\times S^2\smallsetminus P(J)$ to $S^1\times S^2\smallsetminus \G(P(J))$, we see that $S^1\times S^2 \smallsetminus \G(P(J))$ 
likewise has a single JSJ-torus, $\G(T)$. Suppose that $\G(P(J))$ is isotopic (and so ambient isotopic) to $P(J)$ via the isotopy $\Phi\colon (S^1\times S^2)\times[0,1]\to S^1\times S^2$. In particular, we know that $\Phi_1(\G(P(J))=P(J)$.
 Observe that $S^1\times S^2\smallsetminus P(J)$ has a 
single JSJ-torus, $\Phi_1(\G(T))$.
  Since JSJ-decompositions are unique up to isotopy, $T$ is isotopic to $\Phi_1(\G(T))$ and so is isotopic to $\G(T)$ itself. The isotopy of $T$ to $\G(T)$ can be extended by the isotopy extension theorem to produce an ambient isotopy of $S^1\times S^2$. We stack this onto the isotopy $\Phi$ and, abusing notation, call the result $\Phi$ as well. This arranges that $\Phi_1(\G(T)) = T$. Note that $\Phi_1(\G(P(J)))$ may no longer be $P(J)$, but this will not be relevant in our analysis.

Notice that $T$ separates $S^1\times S^2$ into the (open) solid torus $f_J(\mathring{V})$ and another piece not diffeomorphic to $\mathring{V}$, by conditions (3) and (4). Similarly, $\G(T)$ separates $S^1\times S^2$ into the open solid torus $\G(f_J(\mathring{V}))$ and another piece not diffeomorphic to $\mathring{V}$.  It follows that $\Phi_1(\G(f_J(V))) = f_J(V)$.  Thus, $\Phi_1\circ \G\circ f_J$ and $f_J$ give two (possibly different) framings of $J$.
Since $\G\circ f_J$ gives a framing of $\G(J)$, we conclude that $\Phi$ gives an ambient isotopy from $J$ to $\G(J)$, as we claimed.\end{proof}

\begin{proposition}\label{prop:winding0} If $P(J)$ is the winding number $0$ knot in Figure~\ref{fig:satellite}, then $P(J)$ is not isotopic to $\G(P(J))$. 
\end{proposition}
\begin{proof} Suppose $P(J)$ is isotopic to $\G(P(J))$. 
We verify the conditions in Proposition~\ref{prop:satellite}. Note that $P$ is the famous Whitehead doubling pattern. The complement of $P$ in $V$ is the complement of the Whitehead link in $S^3$. The Whitehead link is well-known to be hyperbolic (see e.g.~\cite[Proposition~7.4]{purcell-book}), meaning that the complement admits a complete hyperbolic metric, which implies that it is atoroidal (see e.g.~\cite[Theorem~8.13]{purcell-book}). It is a classical fact that $\pi_1(\partial V)\to \pi_1(V\smallsetminus P)$ is injective; see e.g.~\cite[Sections~3G and 4D]{Rolfsen:1976-1}.

The Mazur link, otherwise known as L7a6, is also hyperbolic~\cite{linkinfo}, but this does not suffice to see that the Mazur knot remains hyperbolic in $S^1\times S^2$. Note that the Mazur knot is obtained by performing 0-framed Dehn surgery on one of the two components of the Mazur link, which happens to be symmetric. A routine Kirby calculus argument shows that the complement of the Mazur knot in $S^1\times S^2$ is the result of $(-1,-3, -\tfrac{1}{2})$ surgery on three of the four components of the minimally twisted four-chain link, and this manifold is known to be hyperbolic~\cite[Corollary 3.6]{MPR}. As before, this implies that $(S^1\times S^2)\smallsetminus \mathring\nu(J)$ is atoroidal.

That the Mazur knot $J\subseteq S^1\times S^2$ is not isotopic to the Hopf knot is shown in~\cite[Proposition~2.8]{DNPR} (see also~\cite[Section~4]{livingston-wrappingnumber}). Since it has winding number 1, we see that $0\neq 1=[J]\in H_1(S^1\times S^2;\Z)$. 

It remains to show that the inclusion induced map $\pi_1(\partial \nu(J))\to \pi_1((S^1\times S^2)\smallsetminus \mathring\nu(J))$ is injective. By the symmetry of the Mazur link, we can switch the two components of the diagram in the middle panel of Figure~\ref{fig:satellite}. In other words, we consider instead the case where the 0-framed surgery is performed on the knot labelled $J$ and ask whether the boundary of the solid torus obtained as the exterior of the second component is $\pi_1$-injective into the surgered manifold. This follows from~\cite[Theorem~1.1(c)]{gabai:surgery-knots-solidtori}. To apply the theorem, we need to observe that the Mazur knot is not a torus knot, nor a 1-bridge braid, and that $H_1((S^1\times S^2)\smallsetminus J;\Z)\cong \Z$.

Since all the conditions are met, by Proposition~\ref{prop:satellite}, we conclude that $J$ is isotopic to $\G(J)$. This gives us a contradiction since by Proposition~\ref{prop:w=-2k} any odd winding number knot which is not isotopic to the Hopf knot changes its isotopy class by the Gluck twist. This completes the proof.
\end{proof}

Finally, we prove Theorem~\ref{thm:glucktwist}.

\begin{reptheorem}{thm:glucktwist}For each integer $w$, there exists a knot $K \subseteq S^1\times S^2$ with winding number $w$ such that $K$ is not isotopic to $\G(K)$.

Moreover, if $K\subseteq S^1\times S^2$ has odd winding number and is isotopic to $\G(K)$, then $K$ has geometric winding number $1$.
\end{reptheorem}

\begin{proof}The first sentence follows from Propositions~\ref{prop:w=-2k}, \ref{prop:evenwinding}, \ref{prop:winding2}, and~\ref{prop:winding0}. The second sentence follows from Proposition~\ref{prop:w=-2k}. \end{proof}

\begin{remark}
Proposition~\ref{prop:satellite} could be used to prove the first part of Theorem~\ref{thm:glucktwist} for all winding numbers. We have already shown the winding number $0$ case. For winding number $p=\pm 1$, let $P$ denote the Mazur knot in $S^1\times D^2$ and for $|p|\geq 2$, let $P$ denote the $(p,1)$ torus braid in $S^1\times D^2$. Let $J$ be the knot shown in Figure~\ref{fig:satellite}. Then $P(J)$ is not isotopic to $\G(P(J))$. But we do not see how to recover the second part of Theorem~\ref{thm:glucktwist} by this method. We have chosen to include the results of Section~\ref{sec:even} since the examples of Figure~\ref{fig:evenwinding} are so simple.
\end{remark}


\begin{appendix}

\section{The Kauffman bracket skein module of \texorpdfstring{$S^1\times S^2$}{S1xS2}}\label{appendixA}

In~\cite{Przytycki-1991}, Przytycki defined a $\Z[A^{\pm1}]$-module, $\S(M)$, associated to a 3-manifold $M$ as a quotient of the free module generated by isotopy classes of framed unoriented links in $M$. This is inspired by the skein relation for the classical Kauffman bracket \cite{Kauffman-1990}. In this section, we prove Theorem~\ref{thm:Gluck-and-skein} indicating that this tool, while powerful, cannot be used to differentiate the isotopy class of a knot in $S^1\times S^2$ from its image under a Gluck twist. Recall that given a framed knot $\K$, the framed knot obtained by adding $a$ full right-handed twists to the framing is denoted by $\K^a$.  Throughout this section, all links are framed and unoriented. In particular, winding numbers are only defined modulo $2$.

First we recall the precise definition as well as the structure of $\S(S^1\times S^2)$ appearing in \cite{HoPr-1995}.  
       Let $\mathscr{M}$ be the free $\Z[A^{\pm1}]$-module generated by isotopy classes of framed links in $S^1\times S^2$. Let $\mL$ be a framed link and $\U$ be the 0-framed unknot. The skein 
module, $\S(S^1\times S^2)$, is the quotient of $\mathscr M$ given by setting 
\[
[\mL \sqcup \U] = -(A^{-2}+A^2)[\mL]
\]
and 
\[
[\mL] = A[\mL_0]+A^{-1} [\mL_\infty]
\]
where $\mL$, $\mL_0$, and $\mL_\infty$ are identical outside of a small ball within which they are as in Figure~\ref{fig:skein} and $\sqcup$ denotes a split union.
It is a consequence of these relations that if $\mL$ and $\mL^1$ differ by changing the framing by a full right-handed twist about the meridian of any one component of $\mL$, as in Figure~\ref{fig:L1}, then 
\begin{equation}\label{eqn:gluck-framing}
[\mL^1] = -A^3[\mL].
\end{equation} 
The Gluck twist $\G$ naturally acts on $\S(S^1\times S^2)$ and by abuse of notation we will denote this action by $\G$ as well.
\begin{figure}[htb]
\begin{tikzpicture}
\node at (0,0) {\includegraphics[width=.4\textwidth]{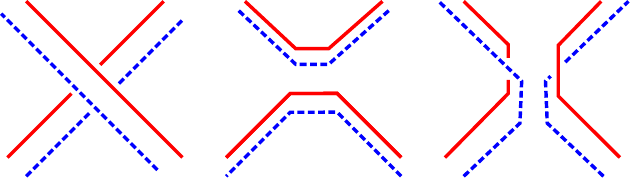}};
\node at (0,-1.2) {{$\mL_0$}};
\node at (-2.4,-1.2) {{$\mL$}};
\node at (2.4,-1.2) {{$\mL_\infty$}};
\end{tikzpicture}
\caption{The skein relation for the module $\S(S^1\times S^2)$. The links are shown in red and the framings are denoted by the push-offs shown in blue.}\label{fig:skein}
\end{figure}
\begin{figure}[htb]
\begin{tikzpicture}
\node at (0,0) {\includegraphics[width=.4\textwidth]{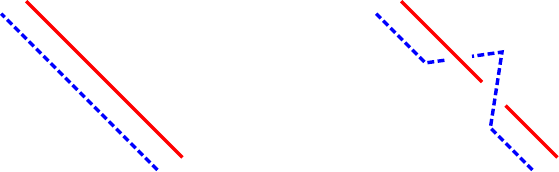}};

\node at (-2,-1.2) {{$\mL$}};
\node at (2,-1.2) {{$\mL^1$}};
\end{tikzpicture}
\caption{The framed link $\mL$.  Changing the framing by a twist about a meridian $\mL^1$.  }\label{fig:L1}
\end{figure}

There is a natural generating set of  $\S(S^1\times S^2)$ given by $z^0$ (the empty diagram), $z$, $z^2$, \ldots, depicted in Figure~\ref{fig:z}. In fact, $\S(S^1\times S^2)$ has a $\Z[A^{\pm1}]$-algebra structure, whose multiplication is given by defining $z^i\cdot z^j := z^{i+j}$ for all $i,j$, and then extending linearly.
Two new generating sets are used in \cite{HoPr-1995}.  The first is given by $$e_0=z^0, e_1 = z^1, \text{ and }e_i = z\cdot e_{i-1} -e_{i-2} \text{ for }i\ge 2.$$  
Note that the recursive step in this formula uses the $\Z[A^{\pm1}]$-module structure on $\S(S^1\times S^2)$.
The latter generating set consists of eigenvectors of $\G$, even regarded in $\S(S^1\times D^2)$.  More precisely, by \cite[Equation(1), p.~67]{HoPr-1995}
\begin{equation}\label{eqn:eigenvector}
\G(e_i) = (-1)^i A^{i^2+2i} e_i.
\end{equation}
\begin{figure}[htb]
\begin{tikzpicture}
\node at (0,0) {\includegraphics[width=.3\textwidth]{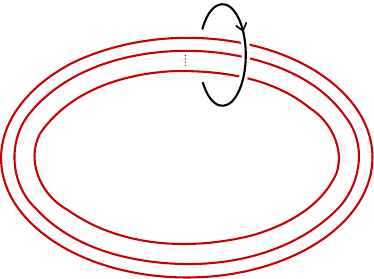}};
\node at (1.2,1.6) {$0$};
\end{tikzpicture}
\caption{A diagram for the $n$-component link $z^n$, with the blackboard framing. Together with the empty diagram $z^0$, the set $\{z^i\}_{i\geq 0}$ forms a generating set for $\S(S^1\times S^2)$.}\label{fig:z}
\end{figure}

The second generating set is given by $e_0' = e_0$, $e_1'=e_1$, $e_2' = e_2$, and
\begin{equation}\label{eqn:recurrence-gen-set}
e_i' = e_i+e_{i-2}'\text{ for }i\ge 3.
\end{equation} 
Hoste-Przytycki prove that for $i\ge 1$, the element $e_i'$ generates the $\frac{\Z[A^{\pm1}]}{(1-A^{2i+4})}$-summand in the following direct sum decomposition~\cite[Theorem 3]{HoPr-1995}.
\begin{equation}\label{eqn:decomp}
\S(S^1\times S^2)\cong \Z[A^{\pm1}] \oplus \bigoplus_{i=1}^\infty \frac{\Z[A^{\pm1}]}{(1-A^{2i+4})}.
\end{equation}
In particular, this means that $e_i'=A^{2i+4}e_i'$ for each $i\geq 1$.
Our analysis will require the following pair of observations.

\begin{proposition}\label{prop:Skein-eigen}
For all $i\geq 0$, we have $$\G(e_{2i}') = e_{2i}' \text{ and }\G(e_{2i+1}') = -A^{2i+3}e_{2i+1}'.$$
\end{proposition}
\begin{proof}
We give an inductive argument. The base cases are obvious. Next we consider $e_{2i+1}'$ for $i\geq 1$.  By the formula~\pref{eqn:recurrence-gen-set} for $e_i'$ in terms of $e_i$,
$$
\G(e_{2i+1}') = \G(e_{2i+1}+e_{2i-1}').
$$
By linearity, the inductive assumption, and \pref{eqn:eigenvector}
$$
\G(e_{2i+1}') = -A^{4i^2+8i+3}e_{2i+1}-A^{2i+1}e_{2i-1}'.
$$
Next, since $e_i' = e_i+e_{i-2}'$, it follows that $e_i = e_i'-e_{i-2}'$ for $i\geq 3$.  Thus,
\begin{align*}
\G(e_{2i+1}') &= -A^{4i^2+8i+3}e_{2i+1}'+(A^{4i^2+8i+3}-A^{2i+1})e_{2i-1}'.
\end{align*}
By~\cite[Theorem 3]{HoPr-1995} (see~\eqref{eqn:decomp}), we have $$A^{4i+6}e_{2i+1}' = e_{2i+1}'  \text{ and  } A^{4i+2}e_{2i-1}' = e_{2i-1}'.$$ By rewriting the equation above, we see that
\begin{align*}
\G(e_{2i+1}') &= -A^{(4i+6)i+2i+3}e_{2i+1}'+(A^{(4i+2)(i+1)+2i+1}-A^{2i+1})e_{2i-1}'
\\&=-A^{2i+3}e_{2i+1}'+(A^{2i+1}-A^{2i+1})e_{2i-1}'
\\&=-A^{2i+3}e_{2i+1}'
\end{align*}
which completes the inductive step of the proof.  The proof for the even case is virtually identical, again using~\eqref{eqn:decomp}. \end{proof}

\begin{proposition}\label{prop:skein-split}
If $\K$ is a framed knot in $S^1\times S^2$ with geometric winding number $w$, then $$[\K] = \Sum_{i=0}^w a_i e_i' \in \S(S^1\times S^2) $$ where $a_i\in\Z[A^{\pm 1}]$ for all $i$ and $a_i=0$ whenever $i \centernot{\equiv} w \pmod{2}$. 
\end{proposition}
\begin{proof}
Recall from Section~\ref{sec:glucktwist} that a diagram for $\K$ is a knot $K$ in the complement of the $0$-framed unknot $U$. Now fix a classical diagram for $K\cup U$, namely a pair of immersed curves $D_\K\cup D_U$ in the plane with crossing information.  As $U$ is the unknot, we arrange that $D_U$ has no self-crossings.  We now resolve every self-crossing of $D_\K$ using the skein relation. In doing so we see that $[\K]\in \S(S^1\times S^2)$ is equivalent to a linear combination of classes of links each of which has no self crossings when thought of as links in the complement of $D_U$,  has the same winding number as $\K$ and has geometric winding number less than or equal to the geometric winding number of $\K$.  In other words $$[\K] = \Sum_{i=1}^w b_i z^i \in \S(S^1\times S^2)$$ where $b_i\in \Z[A^{\pm 1}]$ for all $i$, and $b_i = 0$ whenever $i \centernot{\equiv} w\pmod{2}$.

Using the recurrence relations defining $e_i$ in terms of $z^i$ and $e_i'$ in terms of $e_i$ it is now straightforward to check by induction that $$z^i = \Sum_{j=1}^i c_j e_j'$$ where $c_j=0$
whenever $j \centernot{\equiv} i \pmod{2}$. This completes the proof.
\end{proof}

We are now ready to prove Theorem~\ref{thm:Gluck-and-skein}:
\begin{reptheorem}{thm:Gluck-and-skein}
Let $\K$ be a framed knot in $S^1\times S^2$ with geometric winding number $w$ and let $[\K]$ be its class in $\S(S^1\times S^2)$. If $w$ is even then $[\K] = [\G(\K)]$. If $w$ is odd, then $[\K]=[\G(\K)^f]$ for some $f\in \Z$.

In particular, the class in the Kauffman skein bracket module does not distinguish the knot type of $K$ from that of $\G(K)$.
\end{reptheorem}

\begin{proof}
If $\K$ has even geometric winding number $w$, then $$[\K] = \Sum_{i=0}^{w/2} a_i e_{2i}' \in \S(S^1\times S^2)$$ 
where
 $a_i\in\Z[A^{\pm 1}]$ for all $i$ by Proposition~\ref{prop:skein-split}.  By Proposition~\ref{prop:Skein-eigen}, we have 
\[
[\G(\K)]=\G([\K])=\G\left(\Sum_{i=0}^{w/2} a_i e_{2i}'\right)=\Sum_{i=0}^{w/2} a_i \G(e_{2i}')=\Sum_{i=0}^{w/2} a_i e_{2i}'=
[\K] \in \S(S^1\times S^2).
\]
If $\K$ has odd geometric winding number $w$, then 
\[
[\K] = \Sum_{i=0}^{(w-1)/2} a_i e_{2i+1}'\in \S(S^1\times S^2)
\]
by Proposition~\ref{prop:skein-split}.  By Proposition~\ref{prop:Skein-eigen}, 
\begin{align*}
[\G(\K)] =\G([\K])=\G\left(\Sum_{i=0}^{(w-1)/2} a_i e_{2i+1}'\right) &=\Sum_{i=0}^{(w-1)/2} a_i \G(e_{2i+1}')\\
&=\Sum_{i=0}^{(w-1)/2} -A^{2i+3} a_i e_{2i+1}' \in \S(S^1\times S^2).
\end{align*}
Let $f$ be the least common multiple of $\{2i+3\}_{i=0}^{(w-1)/2}$.  This choice implies that for all $i=0,\dots, \frac{w-1}{2}$, we have that $3f$ is an odd multiple of $2i+3$. Hence $3f+2i+3\equiv 0 \pmod{4i+6}$ for all $i$. Recall that $A^{4i+6} e_{2i+1}' = e_{2i+1}'$ from~\eqref{eqn:decomp}. Therefore in $\S(S^1\times S^2)$ we have 
\[
[\G(\K)^f] =\Sum_{i=0}^{(w-1)/2} A^{3f}A^{2i+3} a_i e_{2i+1}' = \Sum_{i=0}^{(w-1)/2} A^{3f+2i+3} a_i e_{2i+1}' = \Sum_{i=0}^{(w-1)/2} a_i e_{2i+1}'
\]
using~\eqref{eqn:gluck-framing}.
We conclude that
$$
[\G(\K)^f] =\Sum_{i=0}^{(w-1)/2} a_i e_{2i+1}' = [\K] \in \S(S^1\times S^2),
$$
completing the proof. \end{proof}

\end{appendix}

\bibliographystyle{alpha}
\bibliography{research.bib}
\end{document}